\documentclass[reqno]{amsart}
\usepackage[usenames,dvipsnames]{xcolor}
\usepackage{amsfonts, fancyhdr, amsmath, amssymb, hyperref, url, amsthm, subfigure, tikz-cd, verbatim, enumitem,lscape, todonotes, varioref}
\usetikzlibrary{decorations.pathmorphing}
\usepackage[mathscr]{eucal} 
\usepackage[noabbrev, capitalise, nameinlink]{cleveref}

\definecolor{crimson}{rgb}{0.86, 0.08, 0.24}
\definecolor{darkcyan}{rgb}{0.0, 0.55, 0.55}

\hypersetup{
colorlinks={true},
linkcolor= {darkcyan},
citecolor= darkcyan}

\numberwithin{equation}{section}

\newtheorem{thm}{Theorem}[section]
\newtheorem{lem}[thm]{Lemma}

\newtheorem{prop}[thm]{Proposition}
\newtheorem{df}[thm]{Definition}
\newtheorem*{df*}{Definition}

\newtheorem{cor}[thm]{Corollary}
\newtheorem{rmk}[thm]{Remark}
\newtheorem{construction}[thm]{Construction}
\newtheorem{exam}[thm]{Example}
\newtheorem{quest}[thm]{Question}

\newtheorem{thmx}{Theorem}

\newcommand{\EF}{\widetilde{E}\mathcal{F}}

\newcommand{\MUR}{MU_\mathbb{R}}

\newcommand{\MUG}{MU^{(\!(G)\!)}}
\newcommand{\BPR}{BP_\mathbb{R}}
\newcommand{\BPCfour}{BP^{(\!(C_{4})\!)}}

\newcommand{\BPQeight}{BP^{(\!(Q_{8})\!)}}
\newcommand{\BPCtwoCtwo}{BP^{(\!(C_2 \times C_2)\!)}}
\newcommand{\BPCfourone}{BP^{(\!(C_{4})\!)}\langle 1 \rangle}
\newcommand{\BPCn}{BP^{(\!(C_{2^n})\!)}}

\newcommand{\Mpi}{{\underline{\pi}}}

\newcommand{\Cn}{C_{2^n}}

\newcommand{\FF}{\mathcal{F}}

\newcommand{\E}{\underline{\mathcal{E}}}
\newcommand{\sm}{\wedge}

\newcommand{\Tr}{Tr}

\newcommand{\im}{\textup{im}}

\newcommand{\SliceSS}{\textup{SliceSS}}
\newcommand{\LSliceSS}{\textup{LSliceSS}}

\title[A stratification of the equivariant slice filtration]{A stratification of the equivariant slice filtration}

\author{Lennart Meier}
\address{Mathematical Institute, Utrecht University, Utrecht, 3584 CD, the Netherlands}
\email{f.l.m.meier@uu.nl}

\author{XiaoLin Danny Shi}
\address{Department of Mathematics, 
University of Washington, 
4110 E Stevens Way NE,
Seattle, WA 98195
}
\email{dannyshixl@gmail.com}

\author{Mingcong Zeng}
\address{Department of Mathematics, 
Ny Munkegade 118, 
8000 Aarhus C, 
Denmark}
\email{mingcongzeng@gmail.com}

\begin{document}
%%%%%%%%%%%Abstract

\maketitle
\begin{abstract}
In this paper, we construct a stratification tower for the equivariant slice filtration.  This tower stratifies the slice spectral sequence of a $G$-spectrum $X$ into distinct regions.  Within each of these regions, the differentials are determined by the localized slice spectral sequences, which compute the geometric fixed points along with their associated residue group actions.  Consequently, the stratification tower offers an inductive method of understanding the entirety of the equivariant slice spectral sequence of $X$ by examining each of its distinct stratification regions. 
\end{abstract}

{\hypersetup{linkcolor=black}
\tableofcontents}

\section{Introduction}

The equivariant slice filtration was introduced by Hill, Hopkins, and Ravenel as a central component of their landmark solution to the Kervaire invariant one problem \cite{HHR}.  This filtration gives rise to the slice spectral sequence, a powerful tool for analyzing the equivariant homotopy groups of norms of the Real bordism spectrum $\MUG := N_{C_2}^{G} \MUR$.  

Beyond its significance in resolving the Kervaire invariant one problem, the equivariant slice spectral sequence has proven to be a highly effective computational tool in equivariant homotopy theory.  It has led to computational advancements in chromatic homotopy theory \cite{HahnShi, BHSZ, BHLSZ, HSWX}.

When the group is $C_2$, the $C_2$-equivariant slice filtration was originally introduced by Dugger \cite{DuggerKR}.  Dugger employed this filtration to compute Atiyah's Real $K$-theory $K_\mathbb{R}$ \cite{AtiyahKR}.  Beyond exhibiting both the complex Bott periodicity of $KU$ and the real Bott periodicity of $KO$, the $C_2$-equivariant slice spectral sequence of $K_\mathbb{R}$ is notably concentrated within the first and third quadrants, further bounded by lines with slopes $0$ and $1$.

In \cite{HHRKH}, Hill, Hopkins, and Ravenel presented the first comprehensive computation of a $C_4$-equivariant slice spectral sequence.  Their computation resulted in a complete computation of the height-2 Lubin--Tate theory, equipped with a $C_4$-action originating from the Morava stabilizer group (see also \cite{BehrensOrmsby, BBHS}).  During their computation, an intriguing pattern emerged: the $C_4$-slice spectral sequence is stratified into two distinct regions within both the first and third quadrants (see \cref{fig:HHRE2C4SS}).  These regions are bounded by lines with slopes $0$, $1$, and $3$, and they exhibit significant differences in the distribution of classes and the behavior of differentials.

\begin{figure}
\begin{center}
\makebox[\textwidth]{\hspace{0in}\includegraphics[page = 1, trim={4.93cm 7.55cm 8.4cm 2.82cm}, clip, scale = 0.8]{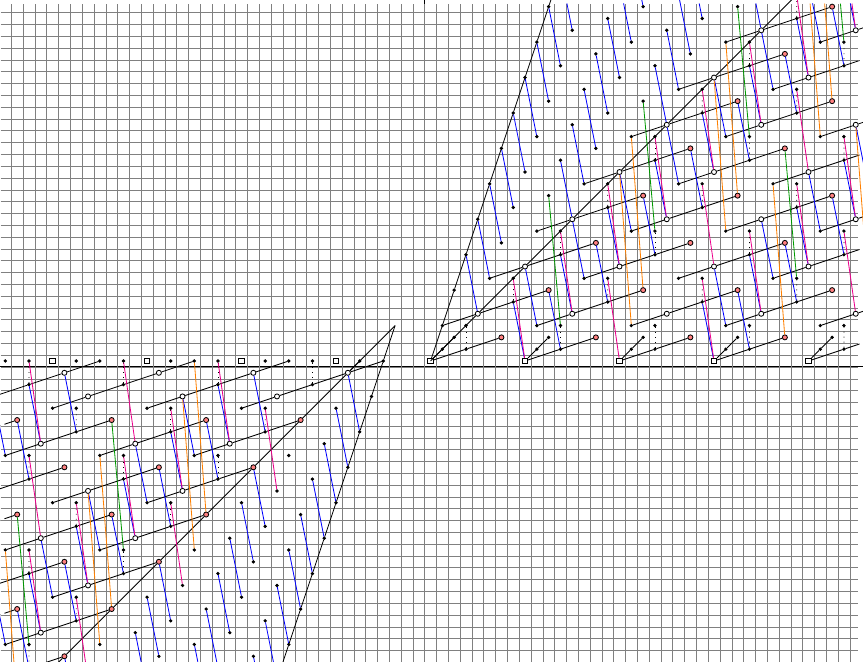}}
\caption{Hill--Hopkins--Ravenel's computation of the $C_4$-slice spectral sequence of $E_2^{hC_4}$.}
\hfill
\label{fig:HHRE2C4SS}
\end{center}
\end{figure}

More generally, the theory of Real orientations enables the utilization of the equivariant slice spectral sequence to compute the fixed points of Lubin--Tate theories via quotients of $\MUG$, as explored in \cite{HahnShi, BHSZ, HSWX}.  In \cite{HSWX}, the $C_4$-equivariant slice spectral sequence of a height-$4$ Lubin--Tate theory is computed.  Despite its greater complexity compared to Hill--Hopkins--Ravenel's height-$2$ computation, it exhibits analogous structural behavior.  The slice spectral sequence remains stratified into two distinct regions, bounded by lines with slopes $0$, $1$, and $3$, in the same manner as observed in Hill--Hopkins--Ravenel's computation.

\vspace{0.05in}

To this end, the motivations for this paper are as follows: 
\begin{enumerate}
\item Is the equivariant slice spectral sequence of a $G$-spectrum $X$ always stratified into distinct regions? 
\item If the answer to (1) is affirmative, then is there a systematic method to individually analyze the differentials within each region and subsequently amalgamate this information to reconstruct the entire slice spectral sequence?
\end{enumerate}

\begin{thmx}[Stratification]\label{thm:introThm1IntegerStratification}
Let $G$ be a finite group and $X$ a $G$-spectrum.  There is a decreasing filtration of the equivariant slice spectral sequence of $X$ given by the tower 
\[\left\{\EF_{\leq \bullet} \wedge \SliceSS(X), \, 0 \leq \bullet \leq |G| \right\} \]
of localized slice spectral sequences, where $\EF_{\leq 0} \wedge \SliceSS(X) \simeq \SliceSS(X)$ and $\EF_{\leq |G|} \wedge \SliceSS(X) \simeq *$.  This filtration has the property that for all $0 \leq h \leq |G|$, the map 
\[\varphi_h: \SliceSS(X) \longrightarrow \EF_{\leq h} \wedge \SliceSS(X)\]
induces an isomorphism of $RO(G)$-graded spectral sequences on or above a line of slope $(h-1)$ within the region where $t-s \geq 0$. 
\end{thmx}
 
In \cref{thm:introThm1IntegerStratification}, $\mathcal{F}_{\leq h}$ is the family consisting of all subgroups of $G$ or order $\leq h$, as defined in \cref{df:orderFamilies}.  We refer to the tower in \cref{thm:introThm1IntegerStratification} as the \textit{stratification tower}, as it effectively stratifies the slice spectral sequence of $X$ into distinct regions on each $RO(G)$-graded page.  Each of these individual regions offers greater manageability than dealing with the entire slice spectral sequence.  In particular, when $G = C_{p^n}$, the equivariant slice spectral sequence of a $C_{p^n}$-spectrum $X$ will be stratified into $n$ distinct regions, separated by lines of slopes $0$, $(p-1)$, $(p^2-1)$, $\ldots$, and $(p^{n}-1)$.

\subsection{The stratification tower} The construction of the stratification tower in \cref{thm:introThm1IntegerStratification} relies on the %utilization of the 
localized slice spectral sequence, as studied by the authors in \cite{MeierShiZengHF2}.  

Let $G$ be a finite group, $X$ a $G$-spectrum, and $\FF$ a family of subgroups of $G$.  The localized slice spectral sequence of $X$ with respect to $\mathcal{F}$ is the spectral sequence associated with the tower $\EF \wedge P^\bullet X$, obtained by smashing the universal space $\EF$ with the slice tower of $X$.  This spectral sequence is denoted by $\EF \wedge \SliceSS(X)$.  

Whenever there is an inclusion $\mathcal{F} \subset \mathcal{F}'$ of families, it induces a map $\EF \to \EF'$ of universal spaces, consequently leading to a map 
\[\varphi: \EF \wedge \SliceSS(X) \longrightarrow \EF' \wedge \SliceSS(X)\]
of the corresponding localized slice spectral sequences.  At the core of analyzing the stratification tower in \cref{thm:introThm1IntegerStratification} is the following theorem, which compares two localized slice spectral sequences and establishes isomorphism regions between them. 

\begin{thmx}[Comparison, \cref{thm:ROGSliceIsom2}]\label{thm:IntroThm3Comparison}
For $V \in RO(G)$, the map 
\[\varphi: \EF \wedge \SliceSS(X) \longrightarrow \EF' \wedge \SliceSS(X)\]
induces an isomorphism of spectral sequences on the $(V+t-s, s)$-graded page within the region $\mathcal{R}_{\FF, \FF'}$ that is defined by the following inequalities: 
\[\left\{\begin{array}{ll}
s \geq (|H|_{\max}-1)(t-s) + \tau_V(\FF'\setminus \FF), & \text{when } t-s \geq 0,\\
s \geq (|H|_{\min}-1)(t-s) + \tau_V(\FF'\setminus \FF), & \text{when } t-s < 0.
\end{array} \right.\]
Here, $|H|_{\max}$ and $|H|_{\min}$ represent the orders of the largest and the smallest subgroups in $\FF' \setminus \FF$, respectively.  The constant $\tau_V(\FF' \setminus \FF)$ depends on $V$, $\FF$, and $\FF'$ as defined in \cref{df:HmaxHmintau}.
\end{thmx}

To prove \cref{thm:IntroThm3Comparison}, we examine the fiber of the map $\varphi$ and establish its vanishing regions.  The vanishing region in the fiber will precisely correspond to the isomorphism region of $\varphi$.  Our analysis relies on the novel formulation of slice connectivity introduced by Hill and Yarnall \cite{HillYarnall}, where they prove a crucial connection between slice connectivity and classical connectivity of the geometric fixed points.

The following result is an immediate consequence from \cref{thm:IntroThm3Comparison} by setting $\mathcal{F} = \varnothing$ and $\FF' = \FF_{\leq h}$.  
\begin{thmx}[Slice Recovery Theorem, \cref{thm:SliceRecovery1}]\label{thm:introThmSliceRecovery}
Let $X$ be a $G$-spectrum, and let $h \geq 1$.  On the $(V+t-s, s)$-graded page where $t-s \geq 0$, the map
\[\varphi_h: \SliceSS(X)  \longrightarrow \EF_{\leq h} \sm \SliceSS(X)\]
induces an isomorphism of spectral sequences on or above the line $\mathcal{L}_{h-1}^V$, as defined in \cref{df:LineLVh-1}.  In the region where $t-s < 0$, $\varphi_h$ induces an isomorphism of spectral sequences on or above the horizontal line $s= \tau_V(\mathcal{F}_{\leq h})$.
\end{thmx}

We also extend the vanishing region established by Hill, Hopkins, and Ravenel to the $RO(G)$-graded page. 

\begin{thmx}[Positive Cone, \cref{thm:SliceSSVanishingRegion}] \label{thm:introThm4PositiveCone}
Let $X$ be a connective $G$-spectrum and $V \in RO(G)$.  All the classes on the $(V+t-s, s)$-graded slice spectral sequence of $X$ are concentrated within the conical region defined by the following inequalities: 
\begin{align*}
s & \geq -|V|-k(V)\cdot|G|, \\
s & \leq (|G|-1)(t-s) - |V|+ |G|\cdot \max_{H \subseteq G} |V^H|. 
\end{align*}
Here, $k(V)$ is the smallest non-negative integer such that the representation \linebreak ${k(V) \cdot \rho_G + V}$ is an actual $G$-representation. 
\end{thmx}

In summary, \cref{thm:introThm4PositiveCone} shows that all the classes on the $(V+t-s, s)$-graded page are concentrated within a conical region bounded by a horizontal line and a line of slope $(|G|-1)$.  On the integer-graded page, where $V = 0$ and $k(V) = 0$, this region is the cone bounded by the inequalities $s \geq 0$ and ${s \leq (|G|-1)(t-s)}$ and is commonly referred to as the ``positive cone''. 
 \cref{thm:introThm4PositiveCone} extends the explicit identification of the positive cone in the integer-graded slice spectral sequence by Hill, Hopkins, and Ravenel \cite{HHR} to the $RO(G)$-graded pages.  

\cref{thm:introThm1IntegerStratification} is a consequence of \cref{thm:IntroThm3Comparison} and \cref{thm:introThm4PositiveCone}.  To construct the stratification tower, we consider the poset $\{\EF \wedge \SliceSS(X)\}$ consisting of all localized slice spectral sequences, obtained from the poset $\{\FF\}$ of all families of subgroups of $G$ (ordered by inclusion).  The chain of inclusions 
\[\FF_{\leq 0} \subset \FF_{\leq 1} \subset \cdots \subset \FF_{\leq |G|}\]
induces a corresponding chain of universal spaces, leading to the chain  
\begin{equation}\label{eq:IntroLocSliceSSChain}
\SliceSS(X) \longrightarrow \EF_{\leq 1} \wedge \SliceSS(X)\longrightarrow \cdots \longrightarrow \EF_{\leq |G|-1} \wedge \SliceSS(X) \longrightarrow *
\end{equation}
within the poset $\{\EF \wedge \SliceSS(X)\}$.  This constitutes our stratification tower. 

From our definition, it is evident that the families $\mathcal{F}_{\leq h-1}$ and $\mathcal{F}_{\leq h}$ differ only when there exists a subgroup $H \subset G$ of order $h$.  For such a subgroup, setting $\FF = \varnothing$ and $\mathcal{F}' = \mathcal{F}_{\leq |H|}$ in \cref{thm:IntroThm3Comparison} shows that on the $(V+t-s, s)$-graded page within the region $t-s \geq 0$, the slice spectral sequence of $X$ on or above a line of slope $(|H| -1)$ can be fully recovered from the localized slice spectral sequence $\EF_{\leq |H|} \wedge \SliceSS(X)$.  

Let $1 = h_0 < h_1 < \cdots < h_k < |G|$ denote the orders of subgroups of $G$.  The tower (~\ref{eq:IntroLocSliceSSChain}) stratifies the positive cone in \cref{thm:introThm4PositiveCone} into distinct regions, separated by lines of slope $(h_i-1)$, $1 \leq i \leq k$.  The portion of the slice spectral sequence on or above a line of slope $(h_i-1)$ can be fully recovered from the localized slice spectral sequence $\EF_{\leq h_i} \wedge \SliceSS(X)$ (see \cref{fig:IntroPicStratificationTower}).

\[\begin{tikzcd}
\SliceSS(X) \ar[rr, "0"] \ar[rrd, swap, "h_1-1"] \ar[rrddd, swap, bend right = 7, "h_{k-1}-1"] \ar[rrdddd, bend right = 35, swap, "h_k-1"]&& \EF_{\leq 1} \wedge \SliceSS(X) \ar[d]  \\
&& \EF_{\leq h_1} \wedge \SliceSS(X) \ar[d]  \\
&& \vdots \ar[d] \\ 
&& \EF_{\leq h_{k-1}} \wedge \SliceSS(X) \ar[d] \\
&& \EF_{\leq h_k} \wedge \SliceSS(X) 
\end{tikzcd}\]

\begin{figure}
\begin{center}
\makebox[\textwidth]{\hspace{-0.5in}\includegraphics[trim={0cm 0cm 0cm 0cm}, clip, scale = 0.6]{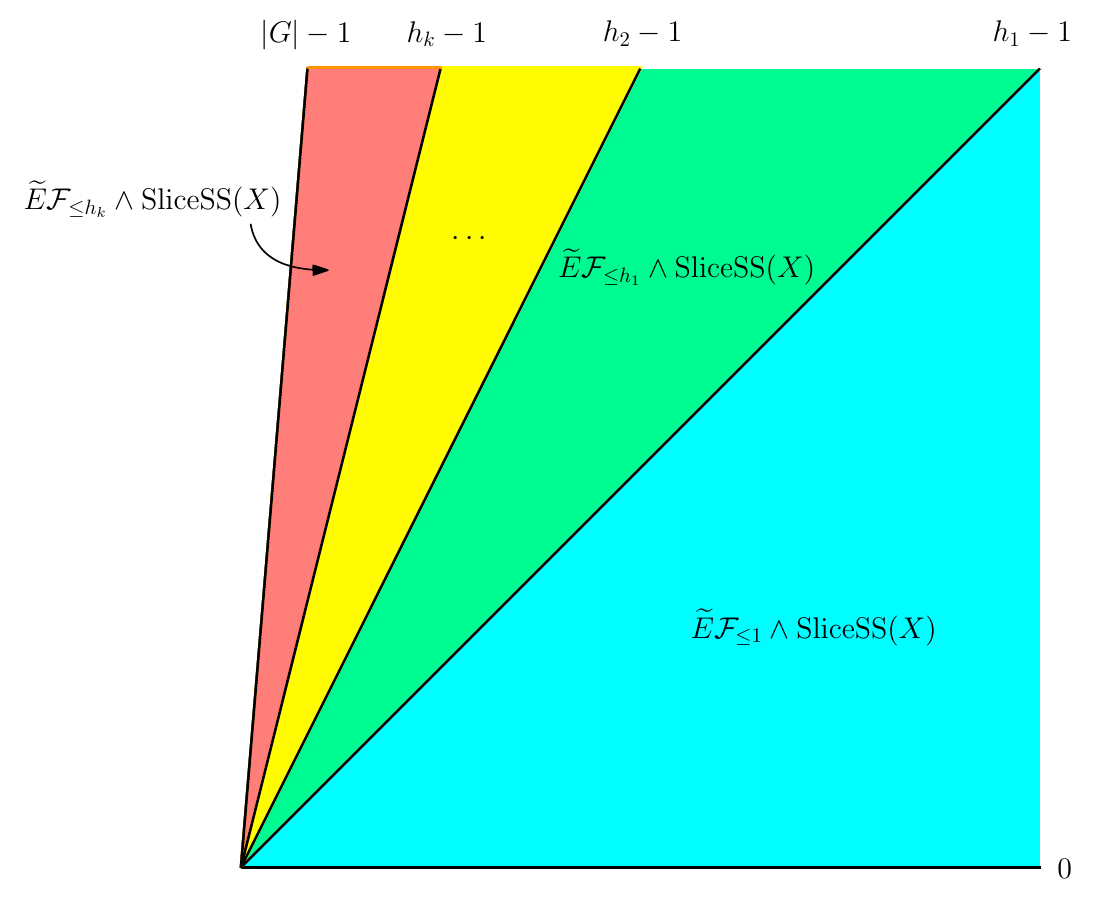}}
\caption{Stratification regions of the slice spectral sequence.}
\hfill
\label{fig:IntroPicStratificationTower}
\end{center}
\end{figure}

\subsection{Applications and further questions}
\subsubsection*{Slice spectral sequences of Lubin--Tate theories}
In \cite{HahnShi}, Hahn and the second author established Real orientations for Lubin--Tate theories.  These Real orientations, combined with the models developed in \cite{BHSZ}, provide a framework for employing the slice spectral sequence to compute fixed points of Lubin--Tate theories via quotients of $\MUG$. 

In \cite{MeierShiZengTranschromatic}, we utilize the stratification tower in \cref{thm:introThm1IntegerStratification} for $G = \Cn$ to analyze the slice spectral sequences of Lubin--Tate theories.  More specifically, we prove that this stratification tower exhibits a type of transchromatic phenomenon, where certain stratification regions in higher height theories will be isomorphic to those in lower height theories through ``shearing isomorphisms".  This approach provides an inductive method for computing higher height theories from lower height theories while also establishing periodicities and vanishing lines in the process. 

\subsubsection*{Stratification of the negative cone}
In this paper, our primary focus is on the positive cone. This is because, for the applications discussed above, the crucial examples are quotients of $\MUG$ and these are connective.  However, it is important to note that in general, when $X$ is not connective, a ``negative cone'' also appears in the slice spectral sequence of $X$, bounded by a horizontal line and a line of slope $(|G|-1)$.    

Given the stratification results established in this paper, a natural question arises: 
\begin{quest}\rm
Can similar stratification results be established for the negative cone?    
\end{quest}  
In an ongoing project with Yutao Liu and Guoqi Yan, we demonstrate that this is indeed possible through an analysis of spectral sequences associated with the towers $F(\EF_{\leq h}, P^\bullet X)$.  We will also introduce a generalized Tate diagram of spectral sequences to explore the interactions between the positive and the negative cones and to identify isomorphism regions among various equivariant spectral sequences.  These include the generalized homotopy orbit spectral sequence, the generalized homotopy fixed point spectral sequence, the generalized Tate spectral sequence, and the (localized) slice spectral sequence.    

\subsubsection*{Stratification of $\BPQeight$}
In contrast to the case when $G=C_{2^n}$, where the equivariant filtration of $\BPCn$ is fully understood by the work of Hill, Hopkins, and Ravenel \cite{HHR}, the $Q_8$-filtration of $\BPQeight:= N_{C_2}^{Q_8}(\BPR)$ remains largely unexplored.
\begin{quest}
What is the $Q_8$-filtration of $\BPQeight$?     
\end{quest}
One promising approach for constructing the appropriate $Q_8$-filtration is to leverage the insights gained from the stratification tower. Specifically, when applied to $Q_8$, the stratification tower suggests that this filtration should be established by combining information from $\BPCtwoCtwo$ (corresponding to the quotient group $Q_8/C_2$) with information from $\BPCfour$ (corresponding to the subgroup $C_4 \subset Q_8$).

\subsubsection*{Poset of spectral sequences}
When working with any equivariant filtration $F^\bullet X$ for a $G$-spectrum $X$, it is always possible to construct the poset $\{\EF \wedge F^\bullet X\}$, from which we obtain a poset of the corresponding spectral sequences.  The focus of this paper has been on the chain ${\{\EF_{\leq i} \wedge F^\bullet X, 0 \leq i \leq |G|\}}$ when ${F^\bullet X = P^\bullet X}$ is the slice filtration, from which we established a stratification of the slice spectral sequence.  

\begin{quest}
What other meaningful phenomena can be extracted from the poset $\{\EF \wedge \SliceSS(X)\}$?
\end{quest}

For example, for $G = Q_8$, the stratification tower we built is
\[\SliceSS(X) \longrightarrow \EF_{\leq 1} \wedge \SliceSS(X) \longrightarrow \EF_{\leq 2} \wedge \SliceSS(X) \longrightarrow \EF_{\leq 4} \wedge \SliceSS(X).\]
To refine this stratification tower further, note that there are three distinct $C_4$-subgroups of $Q_8$, namely $\langle i \rangle$, $\langle j \rangle$, and $\langle k \rangle$.  These subgroups give three families: 
\begin{align*}
\mathcal{F}_{\langle i \rangle} & := \{\langle 1 \rangle, \langle -1 \rangle, \langle i \rangle\} \\
\mathcal{F}_{\langle j \rangle} & := \{\langle 1 \rangle, \langle -1 \rangle, \langle j \rangle\} \\
\mathcal{F}_{\langle k \rangle} & := \{\langle 1 \rangle, \langle -1 \rangle, \langle k \rangle\}.    
\end{align*}
The localized slice spectral sequences corresponding to these families naturally fit within the stratification tower, forming the following diagram:

\[\begin{tikzcd}
& \EF_{\langle i \rangle} \wedge \SliceSS(X) \ar[rd] \\
\EF_{\leq 2} \wedge \SliceSS(X) \ar[ru] \ar[r] \ar[rd] & \EF_{\langle j \rangle} \wedge \SliceSS(X) \ar[r] & \EF_{\leq 4} \wedge \SliceSS(X)\\
& \EF_{\langle k \rangle} \wedge \SliceSS(X) \ar[ru]
\end{tikzcd}\]
It is interesting to analyze how these localized spectral sequences further refine the stratification tower constructed in this paper. 

\begin{quest}
How to extended our analysis to the poset $\{\EF \wedge F^\bullet X\}$ for other equivariant filtrations? 
\end{quest}
Potential candidates include the filtration for the generalized homotopy orbit, homotopy fixed point, and Tate spectral sequences.  Additionally, we may consider the $\mathcal{O}$-slice filtration developed by Mike Hill, which interpolates between the equivariant Postnikov filtration and the slice filtration.

\subsection{Organization of the paper}
In \cref{sec:LocSliceSS}, we analyze maps between localized slice spectral sequences and prove \cref{thm:IntroThm3Comparison}.  In \cref{sec:Stratification}, we establish \cref{thm:introThmSliceRecovery}, construct the stratification tower, and prove \cref{thm:introThm4PositiveCone}.  The main theorem of this paper, \cref{thm:introThm1IntegerStratification}, follows from \cref{thm:introThmSliceRecovery} and \cref{thm:introThm4PositiveCone}.  

In \cref{sec:SpecializationCyclic}, we specialize our results to the case when $G$ is a cyclic $p$-group.  Here, the localized slice spectral sequences in the stratification tower compute geometric fixed points with residue quotient group actions.  This will prove useful in our future applications to Lubin--Tate theories.  

\subsection*{Acknowledgements}
We would like to thank William Balderrama, Agn\`{e}s Beaudry, Mark Behrens, Jeremy Hahn, Mike Hopkins, Hana Jia Kong, Tyler Lawson, Guchuan Li, Wenao Li, Yutao Liu, Doug Ravenel, Vesna Stojanoska, Guozhen Wang, and Guoqi Yan for helpful conversations.  We would like to thank Mike Hill and Zhouli Xu for comments on earlier drafts of our paper.  The first author is supported by the NWO grant VI.Vidi.193.111 and the second author is supported in part by NSF Grant DMS-2313842.

\subsection*{Notations and conventions}
\begin{enumerate}
\item The slice tower we consider will always be the \emph{regular} slice towers (see \cite{UllmanPaper, UllmanThesis}).  
\item For a $G$-spectrum $X$, we denote by $\pi_*^G X$ the integer graded homotopy group of $X$; by $\pi_\star^GX$ the $RO(G)$-graded homotopy group of $X$; by $\Mpi_* X$ the integer-graded Mackey functor valued homotopy group of $X$; and by $\Mpi_\star X$ the $RO(G)$-graded Mackey functor valued homotopy groups of $X$.  
\item All of our spectral sequences are $RO(G)$-graded.  For a fixed $V \in RO(G)$, the ``$(V+t-s,s)$-graded page'' is the part of the spectral sequence consisting of all the classes with degrees $(V+t-s, s)$, where $t, s \in \mathbb{Z}$.  The ``integer-graded page'' is when we set $V = 0$. 

\item For $H\subset G$ and $V \in RO(G)$, $\mathcal{L}_{H, V}$ is the line of slope $(|H|-1)$ on the $(V+t-s, s)$-graded page that is defined by the equation $s = (|H|-1) \cdot (t-s) + |V^H| \cdot |H| - |V|$ (\cref{df:LineLHV}). 

\item For $V \in RO(G)$ and $S$ a collection of subgroups of $G$, $\tau_V(S)$ is the constant $\max_{H \in S} \left(|V^H| \cdot |H| - |V|\right)$ (\cref{df:HmaxHmintau}).

\item For a group $G$, $\mathcal{F}_{\leq h}$ is the family consisting of all subgroups of $G$ of order at most $h$ (\cref{df:orderFamilies}).

\item For a fixed $V \in RO(G)$ and $h \geq 1$, $\mathcal{L}_{h-1}^V$ is the line of slope $(h-1)$ on the $(V+t-s, s)$-graded page that is defined by the equation 
$s = (h-1)(t-s) + \tau_V(\mathcal{F}_{\leq h})$ (\cref{df:LineLVh-1}).

\item For $G$ a finite group and $H \subset G$ a subgroup, $\FF[H]$ is the family consisting of all subgroups of $G$ that do not contain $H$ (\cref{df:FFH}).

\item For $0 \leq i \leq n-1$, $\lambda_{{i}}$ is the 2-dimensional real $C_{p^n}$-representation corresponding to rotation by $\left(\frac{2\pi}{p^{i+1}}\right)$.  In particular, when $p = 2$ and $i = 0$, the representation $\lambda_0$ corresponds to rotation by $\pi$ and is equal to $2\sigma$, where $\sigma$ is the real sign representation of $C_{2^n}$.

\end{enumerate}

%%%%%%%%%%%%%%%%%%%%%%%%%%%
%%%%%%%%%%%%%%%%%%%%%%%%%%%
%%%%%%%%%%%%%%%%%%%%%%%%%%%
\section{The localized slice spectral sequence} \label{sec:LocSliceSS}

% Let $G$ be a finite group and $H \subset G$ a subgroup.  Recall from \cite[Section~3.2]{MeierShiZengHF2} that there is a variant of the slice spectral sequence which computes the equivariant homotopy groups of $\EF[H] \wedge X$.  More precisely, let $\{P^\bullet X\}$ denote the slice tower for $X$.  Then the spectral sequence associated to the tower 
% \[\{Q^\bullet X\} := \{\EF[H] \wedge P^\bullet X\}\]
% is the localized slice spectral sequence ($\LSliceSS$) of $X$.  By \cite[Theorem~3.1]{MeierShiZengHF2}, the localized slice spectral sequence converges strongly to the equivariant homotopy groups of $\EF[H]\wedge X$. 

% In the special case when $G = C_{2^n}$, let $\lambda_{2^{n-i}}$ denote the 2-dimensional real $C_{2^n}$-representation corresponding to rotation by $\left(\frac{\pi}{2^{n-i}}\right)$.  In particular, when $i = n$, $\lambda_1 = 2\sigma$ corresponds to rotation by $\pi$ ($\sigma$ is the real sign representation of $C_{2^n}$).  If $X$ is a $C_{2^n}$-spectrum, then for all $0 \leq i \leq n$, we have an equivalence 
% \[\EF[C_{2^i}] \wedge X \simeq S^{\infty \lambda_{2^{n-i}}} \wedge X \simeq a_{\lambda_{2^{n-i}}}^{-1}X.\]

In this section, we will analyze the localized slice spectral sequence, generalizing the treatment given in \cite[Section~3]{MeierShiZengHF2}.  The main theorem in this section is \cref{thm:ROGSliceIsom2}, which compares localized slice spectral sequences for different families of subgroups of $G$ and prove that they are isomorphic in a range.  

Let $G$ be a finite group.  A \textit{family} $\FF$ of subgroups of $G$ is a set of subgroups that is closed under subconjugacy: if $H \in \FF$ and $g^{-1}Kg \subset H$ for some $g \in G$, then $K \in \FF$.  We will compare localized slice spectral sequences for different families of subgroups $\FF$, obtaining that they agree in a certain region.  

For $X$ a $G$-spectrum, let $P^\bullet X$ denote the slice tower of $X$.  The spectral sequence associated to the tower 
\[Q^\bullet X := \EF \wedge P^\bullet X\]
is the \textit{localized slice spectral sequence} ($\LSliceSS$) of $X$ with respect to $\FF$.  Here, $\widetilde{E}\FF$ is the cofiber of the canonical map $E\FF_+ \to S^0$, and its fixed points are
\[\EF^H \simeq \left\{\begin{array}{ll} * & \text{if } H \in \FF \\
S^0 &\text{if } H \notin \FF. \end{array} \right.\]

Note that the localized slice spectral sequence of $X$ with respect to $\mathcal{F}$ is \textit{not} the slice spectral sequence of $\EF \wedge X$ because $\EF \wedge P^\bullet X \not \simeq P^\bullet(\EF \wedge X)$ in general.  Nevertheless, the localized slice spectral sequence still computes the equivariant homotopy groups of $\EF \wedge X$.  More precisely, the localized slice spectral sequence is a $RO(G)$-graded spectral sequence of Mackey functors that has the form
\[\E_2^{s,V} = \Mpi_{V-s} \EF \sm P^{|V|}_{|V|}X \,\,\, \Longrightarrow \,\,\, \Mpi_{V-s} \EF \wedge X.\]
The proof of \cite[Theorem~3.1]{MeierShiZengHF2} applies to this more general setting to show that this spectral sequence is strongly convergent.  

When we are considering multiple localized slice spectral sequences with respect to different families, we will denote them by $\EF \sm \SliceSS(X)$ to distinguish between them.  In particular, an inclusion $\FF \subset \FF'$ of families will induce a map $\EF \to \EF'$ of spaces, which will further induce a map 
\[\varphi: \EF \sm \SliceSS(X)  \longrightarrow \EF' \sm \SliceSS(X)\]
of the corresponding localized slice spectral sequences.

Recall that a $G$-spectrum is \emph{$(n+1)$-connective} (or \emph{$n$-connected}) if its Mackey functor valued homotopy groups vanish in degrees $\leq n$.  It is called \emph{connective} if it is $0$-connective (or equivalently, $(-1)$-connected).

\begin{prop}\label{prop:E2IsomROG}
 Let $X$ be a $G$-spectrum.  Suppose $\mathcal{F} \subset \mathcal{F}'$ are families of subgroups such that the set $\mathcal{F}' \setminus \mathcal{F}$ consists only of subgroups conjugate to a fixed subgroup $H \subset G$. Then the following statements hold: 
\begin{enumerate}
\item The map 
\begin{equation} \label{eq:mapPhi}
\varphi: \EF \sm \SliceSS(X)  \longrightarrow \EF' \sm \SliceSS(X)
\end{equation}
induces an isomorphism of classes on the $\mathcal{E}_2$-page in the region defined by the inequality  
\[s > |V^H| - \left\lceil \frac{|V|}{|H|} \right\rceil.\]
\item The map $\varphi$ induces a surjection of classes on the $\mathcal{E}_2$-page in the region defined by the equality 
\[s = |V^H| - \left\lceil \frac{|V|}{|H|} \right\rceil.\]
If $H$ is not subconjugate to $K$, this is an isomorphism on the $G/K$-level. 
\item Assume that $H\subset G$ is normal. If $H \subset K$, then the kernel on the $G/K$-level of the surjection from the preceding part is equal to the image of the transfer from the $G/H$-level. 
\end{enumerate}
\end{prop}
\begin{proof}
For (1), let $F$ be the fiber of the map $\EF \to \EF'$, and consider the long exact sequence in homotopy groups induced by the cofiber sequence 
\[F \wedge P^{|V|}_{|V|}X \longrightarrow \EF \wedge P^{|V|}_{|V|}X \longrightarrow \EF' \wedge P^{|V|}_{|V|}X.\]
The map 
\[\Mpi_{V-s}\left(\EF \wedge P^{|V|}_{|V|}X\right) \longrightarrow \Mpi_{V-s}\left(\EF' \wedge P^{|V|}_{|V|}X\right)\]
is an isomorphism if and only if both $\Mpi_{V-s} \left(F \wedge P^{|V|}_{|V|}X\right)$ and $\Mpi_{V-s-1} \left(F \wedge P^{|V|}_{|V|}X\right)$ vanish.  Note that since 
\[\Mpi_{V-s} \left(F \wedge P^{|V|}_{|V|}X\right) = \Mpi_{0}\left(S^{-V+s} \wedge F \wedge P^{|V|}_{|V|}X\right),\]
it suffices to prove that the $G$-spectrum $S^{-V+s} \wedge F \wedge P^{|V|}_{|V|}X$ is $1$-connective in the region defined by the inequality.

Recall that a $G$-spectrum is $n$-connective if and only if its $K$-geometric fixed points for every subgroup $K \subseteq G$ is non-equivariantly $n$-connective (this can be proved by using induction on the order of $|G|$ and the isotropy separation sequence).  Furthermore, by the definition of $F$, 
\[
\Phi^K (F) \simeq \begin{cases}
S^0 & \text{ if } K \text{ is conjugate to }H,\\
* & \text{ else}.
\end{cases}
\]
Since the functor $\Phi^K(-)$ is symmetric monoidal and $\Phi^K(X) \simeq \Phi^H(X)$ whenever $K$ is conjugate to $H$, we only need to show that the spectrum $\Phi^H\left(S^{-V+s} \wedge F \wedge P^{|V|}_{|V|}X\right)$ is $1$-connective.

We have the equivalence 
\begin{align*}
\Phi^H\left(S^{-V+s} \wedge F \wedge P^{|V|}_{|V|}X\right) &= S^{-|V^H| + s} \wedge S^0 \wedge \Phi^H\left(P^{|V|}_{|V|}X\right) \\
&= S^{-|V^H| + s} \wedge \Phi^H\left(P^{|V|}_{|V|}X\right).
\end{align*}
By \cite[Theorem~2.5]{HillYarnall}, the spectrum $\Phi^H\left(P^{|V|}_{|V|}X\right)$ is $\left \lceil \frac{|V|}{|H|} \right \rceil$-connective.  Therefore, $\Phi^H\left(S^{-V+s} \wedge F \wedge P^{|V|}_{|V|}X\right)$ is $1$-connective when 
\[(-|V|^H + s) + \left \lceil \frac{|V|}{|H|} \right \rceil > 0.\]
It follows that $s > |V|^H - \left \lceil \frac{|V|}{|H|} \right \rceil$, as desired. 

For (2), consider again the long exact sequence in homotopy groups induced by the cofiber sequence 
\[F \wedge P^{|V|}_{|V|}X \longrightarrow \EF \wedge P^{|V|}_{|V|}X \longrightarrow \EF' \wedge P^{|V|}_{|V|}X.\]
When $s = |V^H| - \left\lceil \frac{|V|}{|H|} \right\rceil$, then $s+1 > |V^H| - \left\lceil \frac{|V|}{|H|} \right\rceil$ and $\Mpi_{V-s-1}\left(F \wedge P^{|V|}_{|V|}X\right) = 0$ by the proof of (1).  This implies that the map 
\[\Mpi_{V-s}\left(\EF \wedge P^{|V|}_{|V|}X\right) \longrightarrow \Mpi_{V-s}\left(\EF' \wedge P^{|V|}_{|V|}X\right)\]
is surjective.  It remains to identify the kernel of this map.  To do so, it suffices to identify the kernel of the map 
\begin{equation} \label{eqn:E2pagePi0K}
\Mpi_0\left(S^{-V+s} \wedge \EF \wedge P^{|V|}_{|V|}X\right) \longrightarrow \Mpi_0\left(S^{-V+s} \wedge \EF' \wedge P^{|V|}_{|V|}X\right)
\end{equation}
for subgroups $K$ of $G$.

Note that if $H$ is not subconjugate to $K$, 
then $i_K^* \EF = i_K^* \EF'$ and the map \eqref{eqn:E2pagePi0K}
%\[\pi_0^K\left(S^{-V+s} \wedge \EF \wedge P^{|V|}_{|V|}X\right) \longrightarrow \pi_0^K\left(S^{-V+s} \wedge \EF' \wedge P^{|V|}_{|V|}X\right)\]
is an isomorphism.  Therefore, in this case, the kernel is 0.  

For (3), we need to analyze the kernel of the map (\ref{eqn:E2pagePi0K}) when  
$H\subset G$ is normal and $H\subset K$. Without loss of generality, we assume $H\subset K$. 
Since $\mathcal{F}'$ contains $H$, the $H$-spectrum $i_H^* \EF'$ is contractible and so
$\pi_*^H\left(S^{-V+s} \wedge \EF' \wedge P^{|V|}_{|V|}X\right) = 0$.  This implies that the map 
\[\pi_0^H\left(S^{-V+s} \wedge F \wedge P^{|V|}_{|V|}X\right) \longrightarrow \pi_0^H\left(S^{-V+s} \wedge \EF \wedge P^{|V|}_{|V|}X\right)\]
is an isomorphism.  Consider the diagram 
\[
\begin{tikzcd}
\pi_0^K\left(S^{-V+s} \wedge F \wedge P^{|V|}_{|V|}X\right) \ar[r, "\mathbf{1}"] &\pi_0^K\left(S^{-V+s} \wedge \EF \wedge P^{|V|}_{|V|}X\right) \ar[r, twoheadrightarrow, "\mathbf{2}"] & \pi_0^K\left(S^{-V+s} \wedge \EF' \wedge P^{|V|}_{|V|}X\right)  \\ 
\pi_0^H\left(S^{-V+s} \wedge F \wedge P^{|V|}_{|V|}X\right) \ar[r, "\cong"] \ar[u, "\Tr_H^K"] & \pi_0^H\left(S^{-V+s} \wedge \EF \wedge P^{|V|}_{|V|}X\right)  \ar[u, "\Tr_H^K"] &
\end{tikzcd}
\]
Since the top row of the diagram is exact, $\ker(\mathbf{2}) = \im (\mathbf{1})$.  To prove our claim, it suffices to show that the left vertical transfer map 
\[\Tr_H^K: \pi_0^H\left(S^{-V+s} \wedge F \wedge P^{|V|}_{|V|}X\right) \longrightarrow \pi_0^K\left(S^{-V+s} \wedge F \wedge P^{|V|}_{|V|}X\right)\]
is surjective.  

Let $C$ be the cofiber of the map $G/H_+ \to G/K_+$.  The transfer map above is the map 
\[\pi_0^G\left(G/H_+ \wedge S^{-V+s} \wedge F \wedge P^{|V|}_{|V|}X\right) \longrightarrow \pi_0^G\left(G/K_+ \wedge S^{-V+s} \wedge F \wedge P^{|V|}_{|V|}X\right).\]
We will show that $C \wedge S^{-V+s} \wedge F \wedge P^{|V|}_{|V|}X$ is 1-connective.  This can be checked on geometric fixed points for all subgroups, the only relevant case being the $H$-geometric fixed points (because $\Phi^N(F) \simeq *$ when $N$ is not conjugate to $H$).  We have the equality
\[\Phi^H \left(C \wedge S^{-V+s} \wedge F \wedge P^{|V|}_{|V|}X\right) \simeq C^{H} \wedge \Phi^H\left(S^{-V+s} \wedge F \wedge P^{|V|}_{|V|}X\right),\]
where $C^H$ denotes the fixed points of the \textit{space} C.  As proven in part (1), the $H$-geometric fixed points $\Phi^H\left(S^{-V+s} \wedge F \wedge P^{|V|}_{|V|}X\right)$ is 0-connective when we have $s = |V^H| - \left\lceil \frac{|V|}{|H|} \right\rceil$.  For the connectivity of $C^H$, note that $C = G_+ \wedge_K C'$, where $C'$ is the cofiber of of $K$-equivariant map $K/H_+ \to K/K_+$. Since $H\subset G$ is normal, $C^H = \bigvee_{G/K}(C')^H$. The space $(C')^H$ is the cofiber of $(K/H)^H_+ \to S^0$. Since $(K/H)^H =K/H$ is non-empty, this is connected and so is thus $C^H$. It follows that $C^H \wedge \Phi^H\left(S^{-V+s} \wedge F \wedge P^{|V|}_{|V|}X\right)$ is 1-connective, as desired.  
%\lennart{I think, there is a problem with our connectivity analysis of $C$. The space $G_+ \wedge_K C'$ consists of equivalence classes $[(g,c)]$ with $c\in C'$; this is an $H$-fixed point if $[(hg,c)] = [(g,c)]$ for all $h\in H$, i.e.\ if $gHg^{-1} \subset K$ and $c$ is a $gHg^{-1}$-fixed point. Thus, $C^H = \bigvee_{gK: gHg^{-1} \subset K} (C')^{gHg^{-1}}$. The fixed points $(C')^{gHg^{-1}}$ are connected/$1$-connective iff $(K/H)^{gHg^{-1}}$ is non-empty, i.e.\ if $H$ and $gHg^{-1}$ are conjugate \emph{in $K$}. This does not need to hold. E.g.\ take the binary tetrahedral group $G = SL_2(\mathbb{F}_3)$, $K = Q_8$ and $H=C_4$. Then $K$ is normal in $G$ and the group $H$ is normal in $K$, but the different choices of $C_4$ inside $Q_8$ are conjugate in $G$. This kind of nonsense is excluded if $H$ is normal in $G$, which we might add as an assumption.}
\end{proof}

Fix a $V \in RO(G)$ and consider the $(V+t-s, s)$-graded $\mathcal{E}_2$-page, where $t \in \mathbb{Z}$.  \cref{prop:E2IsomROG} shows that the map $\varphi$ (\ref{eq:mapPhi}) induces an isomorphism of classes in the region defined by the inequality 
\[s> |V^H| + t - \left\lceil\frac{|V|+t}{|H|} \right\rceil,\]
and a surjection of classes in the region defined by the equality 
\[s = |V^H| + t - \left\lceil\frac{|V|+t}{|H|} \right\rceil.\]
To better understand what these regions are geometrically, let $\varepsilon := \left\lceil\frac{|V|+t}{|H|} \right\rceil - \frac{|V|+t}{|H|}$.  Rearranging the terms in the relations above, we deduce that the isomorphism region is defined by the inequality 
\[s > (|H|-1) \cdot (t-s) + |V^H| \cdot |H| - |V| - \varepsilon \cdot |H|,\]
and the surjective region is defined by the equality 
\[s = (|H|-1) \cdot (t-s) + |V^H| \cdot |H| - |V| - \varepsilon \cdot |H|.\]
\begin{df} \rm \label{df:LineLHV}
For $H\subset G$ and $V \in RO(G)$, let $\mathcal{L}_{H, V}$ be the line of slope $(|H|-1)$ on the $(V+t-s, s)$-graded page that is defined by the equation
\[s = (|H|-1) \cdot (t-s) + |V^H| \cdot |H| - |V|.\] 
\end{df}
By the definition of $\varepsilon$, we have the inequality $0 \leq \varepsilon \cdot |H| \leq |H|-1$ (the precise value of $\varepsilon \cdot |H|$ depends on $t$ modulo $|H|$).  Our discussion above implies the following result. 

\begin{cor}\label{cor:E2IsomROG}
The map $\varphi$ (\ref{eq:mapPhi}) induces an isomorphism of classes on the $(V+t-s, s)$-graded $\mathcal{E}_2$-page above the line $\mathcal{L}_{H, V}$, and a surjection of classes on the line $\mathcal{L}_{H, V}$. 
\end{cor}
\begin{rmk} \rm
Note that the actual isomorphism region is slightly bigger than the one given in \cref{cor:E2IsomROG}, as it include some classes slightly below the line $\mathcal{L}_{H, V}$, depending on the value of $\varepsilon \cdot |H|$.
\end{rmk}

\begin{df}\rm \label{df:SSisomRegion}
Suppose $\mathcal{E}$ and $\mathcal{E}'$ are two spectral sequences, and $\mathcal{L}$ is a line on the $(V+t-s, s)$-graded page.  A map $\varphi: \mathcal{E} \to \mathcal{E}'$ induces an \textit{isomorphism of spectral sequences} on or above the line $\mathcal{L}$ if the following statements hold: 
\begin{enumerate}
\item On the $\mathcal{E}_2$-page, $\varphi$ induces an isomorphism above the line $\mathcal{L}$, and a surjection on the line $\mathcal{L}$;  
\item For all $r \geq 2$, every nonzero $d_r$-differential in $\mathcal{E}$ originating on or above the line $\mathcal{L}$ is mapped to a nonzero $d_r$-differential in $\mathcal{E}'$;
\item For all $r \geq 2$, every nonzero $d_r$-differential in $\mathcal{E}'$ originating on or above the line $\mathcal{L}$ is the image of a nonzero $d_r$-differential in $\mathcal{E}$.
\end{enumerate}
\end{df}

For the rest of the paper, an \textit{isomorphism of spectral sequences} will always be understood in the sense of \cref{df:SSisomRegion}.

\begin{thm}\label{thm:ROGSliceIsom1}
Let $X$ be a $G$-spectrum and $V \in RO(G)$.  Suppose $\mathcal{F} \subset \mathcal{F}'$ are families of subgroups such that the set $\mathcal{F}' \setminus \mathcal{F}$ consists only of subgroups conjugate to a fixed subgroup $H \subset G$.  Then on the $(V+t-s, s)$-graded page, the map
\[\varphi\colon \EF \sm \SliceSS(X)  \longrightarrow \EF' \sm \SliceSS(X)\]
induces an isomorphism of spectral sequences on or above the line $\mathcal{L}_{H, V}$. 
\end{thm}

\begin{proof}[Proof of \cref{thm:ROGSliceIsom1}]
Statement (1) in \cref{df:SSisomRegion} holds by \cref{cor:E2IsomROG}.  We will abbreviate statement (2) (regarding injectivity) by $I_r$ and statement (3) (regarding surjectivity) by $S_r$.  We will prove by using induction on $r$ that $I_r$ and $S_r$ hold for all $r \geq 2$.  The base case, when $r =2$, is a direct consequence of \cref{prop:E2IsomROG} and \cref{cor:E2IsomROG}.  For the induction step, let $r\geq 3$ and assume that $I_{r'}$ and $S_{r'}$ have already been proved for all $r' < r$. 

Let the line $\mathcal{L}_{H, V}^{r'}$ be the ${r'}$-fold vertical shift of $\mathcal{L}_{H, V}$.  Note that \cref{cor:E2IsomROG}, together with statements $I_{r'}$ and $S_{r'}$, imply that on the $\E_{r'+1}$-page, the map $\varphi\colon \mathcal{E}_{r'+1}^{s,V+t} \to {\mathcal{E}'}_{r'+1}^{s,V+t}$ is a bijection if $(V+t-s, s)$ is above the line $\mathcal{L}_{H, V}^{r'}$, and a surjection if $(V+t-s, s)$ is on the line $\mathcal{L}_{H, V}$.

To prove $I_r$, let $d_r(x) = y$ be a nonzero $d_r$-differential that is originating on or above $\mathcal{L}_{H, V}$ in $\mathcal{E}$.  By the induction hypothesis discussed in the previous paragraph and naturality, $\varphi(y)$ must be killed by a differential of length at most $r$ in $\mathcal{E}'$.  If the length of this differential is $r$, then the source must be $\varphi(x)$ by naturality and we are done.  If this differential is of length $r' < r$, then by $S_{r'}$, there must be a $d_{r'}$-differential in $\mathcal{E}$ whose target is $y$.  This is a contradiction.  

To prove $S_r$, let $d_r(x') = y'$ be a nonzero $d_r$-differential in $\mathcal{E}'$ that is originating on or above the line $\mathcal{L}_{H, V}$.  Let $x$ be a pre-image of $x'$.  By naturality, $x$ must support a differential of length at most $r$.  If this differential is of length exactly $r$, then by naturality and the induction hypothesis, the target of this differential must be $y$, the unique preimage of $y'$.  If this differential of length $r' < r$, then by $I_{r'}$, the class $x'$ must support a differential of length smaller than $r$ as well.  This is a contradiction. 
\end{proof}

We are now ready to compare localized slice spectral sequences for general inclusions $\mathcal{F} \subset \mathcal{F}'$ of families.  

\begin{df}\rm \label{df:HmaxHmintau}
Suppose $V \in RO(G)$, and $S$ is a collection of subgroups of $G$.  Define 
\[\tau_V(S) = \max_{H \in S} \left(|V^H| \cdot |H| - |V|\right).\]
\end{df} 

\begin{thm}[Comparison Theorem]\label{thm:ROGSliceIsom2}
Let $X$ be a $G$-spectrum, and suppose $\mathcal{F} \subset \mathcal{F}'$ is an inclusion of families.  Denote the orders of the largest and the smallest subgroups in $\mathcal{F}'\setminus \mathcal{F}$ by $|H|_{\max}$ and $|H|_{\min}$, respectively.  On the $(V+t-s, s)$-graded page, the map 
\[\varphi: \EF \wedge \SliceSS(X) \longrightarrow \EF' \wedge \SliceSS(X)\]
induces an isomorphism of spectral sequences within the region $\mathcal{R}_{\mathcal{F}, \mathcal{F}'}$ that is defined by the following inequalities (see \cref{fig:PicComparisonThm}):
\[\left\{\begin{array}{ll}
s \geq (|H|_{\max}-1)(t-s) + \tau_V(\mathcal{F}'\setminus \mathcal{F}), & \text{when } t-s \geq 0,\\
s \geq (|H|_{\min}-1)(t-s) + \tau_V(\mathcal{F}'\setminus \mathcal{F}), & \text{when } t-s < 0.
\end{array} \right.\]
\end{thm}

\begin{figure}
\begin{center}
\makebox[\textwidth]{\includegraphics[trim={0cm 0cm 0cm 0cm}, clip, scale = 0.75]{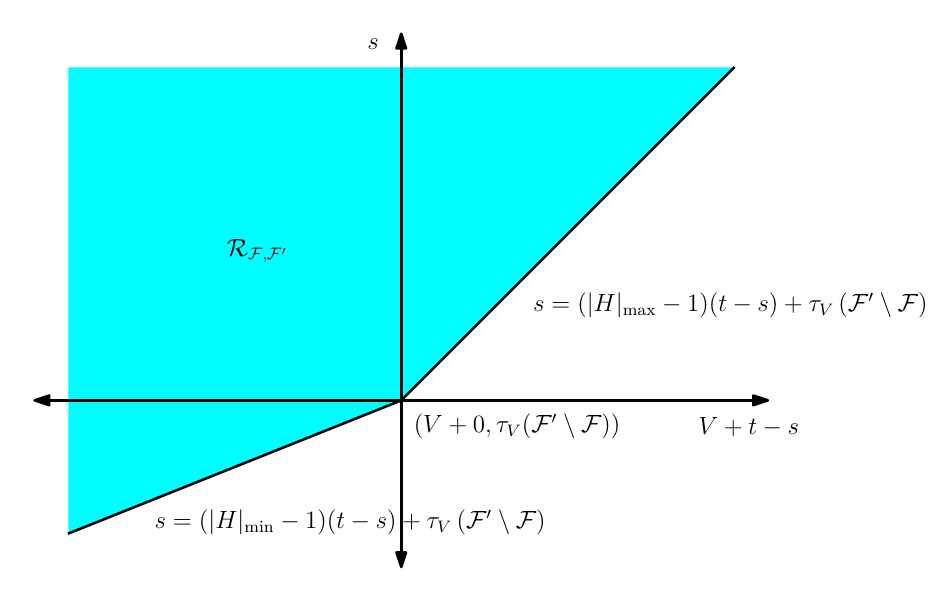}}
\caption{The isomorphism region $\mathcal{R}_{\mathcal{F}, \mathcal{F}'}$ for the map $\varphi$.}
\hfill
\label{fig:PicComparisonThm}
\end{center}
\end{figure}

\begin{proof}
Let 
\[\mathcal{F} = \mathcal{F}_0 \subset \mathcal{F}_1 \subset \cdots \subset \mathcal{F}_n = \mathcal{F}'\]
be a sequence of inclusions such that $\mathcal{F}_i \setminus \mathcal{F}_{i-1}$ consists only of subgroups conjugate to a fixed subgroup $H_i \subset G$.  It is clear that $|H|_{\max} = \max |H_i|$, $|H|_{\min} = \min |H_i|$, and $\tau_V(\mathcal{F}'\setminus \mathcal{F})= \max_{i} |V^{H_i}| \cdot |H_i| - |V|$.

The inclusion of families induces maps of localized slice spectral sequences
\[\EF_0 \wedge \SliceSS(X) \longrightarrow \EF_1 \wedge \SliceSS(X) \longrightarrow \cdots \longrightarrow \EF_n \wedge \SliceSS(X).\]
\cref{thm:ROGSliceIsom1} shows that for $1 \leq i \leq n$, the map 
\[\EF_{i-1} \wedge \SliceSS(X) \longrightarrow \EF_i \wedge \SliceSS(X)\]
induces an isomorphism of spectral sequences on or above the line $\mathcal{L}_{H_i, V}$, which is the region that is defined by the inequality 
\[s \geq (|H_i|-1) \cdot (t-s) + |V^{H_i}| \cdot |H_i| - |V|.\]
The theorem now follows from the fact that the region $\mathcal{R}_{\mathcal{F}, \mathcal{F}'}$ is on or above the line $\mathcal{L}_{H_i, V}$ for all $1 \leq i \leq n$.  
\end{proof}

\begin{rmk}\rm
In the special case when $V = 0$, $\tau_V(\FF'\setminus \FF)$ is equal to 0.  \cref{thm:ROGSliceIsom2} implies that on the integer-graded page, the isomorphism region is the region defined by the inequalities
\[\left\{\begin{array}{ll}
s \geq (|H|_{\max}-1)(t-s), & \text{when } t-s \geq 0,\\
s \geq (|H|_{\min}-1)(t-s), & \text{when } t-s < 0.
\end{array} \right.\]
\end{rmk}

\section{A stratification of the slice spectral sequence}\label{sec:Stratification}

Let $\{\mathcal{F}\}$ be the poset of all families of subgroups of $G$, ordered by inclusion.  As discussed in the previous section, for each inclusion $\mathcal{F} \subset \mathcal{F}'$, there is a map 
\[\EF \wedge \SliceSS(X) \longrightarrow \EF' \wedge \SliceSS(X).\]
Therefore, the poset $\{\FF\}$ produces a poset $\{\EF \wedge \SliceSS(X)\}$ of localized slice spectral sequences.  For each pair of localized slice spectral sequences that is connected by an edge in this poset, the Comparison Theorem (\cref{thm:ROGSliceIsom2}) shows that the two spectral sequences agree in a region.  

In this section, we will discuss a stratification of the entire slice spectral sequence that arises from a chain in this poset.  

\begin{df}\rm \label{df:orderFamilies}
For $0 \leq h \leq |G|$, let $\mathcal{F}_{\leq h}$ be the family consisting of all subgroups of $G$ of order $\leq h$. 
\end{df}
In \cite{UllmanThesis}, the families $\mathcal{F}_{\leq h}$ are called the \textit{order families} of $G$.  Note from definition, $\mathcal{F}_{\leq 0}$ is the empty family, $\mathcal{F}_{\leq |G|-1}$ is the family $\mathcal{P}$ consisting of all proper subgroups of $G$, and $\mathcal{F}_{\leq |G|}$ is the family consisting of all subgroups of $G$.  The chain of inclusions 
\[\FF_{\leq 0} \subset \FF_{\leq 1} \subset \cdots \subset \FF_{\leq |G|}\]
induces the maps
\[\EF_{\leq 0} \longrightarrow \EF_{\leq 1} \longrightarrow \cdots \EF_{\leq |G|-1} \longrightarrow \EF_{\leq |G|}.\]
Since $\EF_{\leq 0} \simeq S^0$ and $\EF_{\leq |G|} \simeq *$, the maps above induces a chain 
\begin{equation}\label{eq:LocalizedSSchain}
\SliceSS(X) \longrightarrow \EF_{\leq 1} \wedge \SliceSS(X)\longrightarrow \cdots \longrightarrow \EF_{\leq |G|-1} \wedge \SliceSS(X) \longrightarrow *
\end{equation}
of the corresponding localized slice spectral sequences.  

Recall from \cref{df:HmaxHmintau} that for $S$ a collection of subgroups of $G$,
\[\tau_V(S) := \max_{H \in S} \left(|V^H| \cdot |H| - |V|\right).\]

\begin{df}\rm\label{df:LineLVh-1}
For a fixed $V \in RO(G)$ and $h \geq 1$, let $\mathcal{L}_{h-1}^V$ represent the line of slope $(h-1)$ on the $(V+t-s, s)$-graded page that is defined by the equation 
\[s = (h-1)(t-s) + \tau_V(\mathcal{F}_{\leq h}).\]
\end{df}

There are two special cases worth noting: 
\begin{enumerate}
\item When $V = 0$, we have $\tau_0(\mathcal{F}_{\leq h})=0$, and $\mathcal{L}_{h-1}^0$ is the line of slope $(h-1)$ through the origin on the integer-graded page.  In this case, we will denote this line by $\mathcal{L}_{h-1}$.  
\item When $h = 1$, we have $\tau_V(\mathcal{F}_{\leq 1}) = |V^e| \cdot 1 - V = 0$.  In this case, the line $L^V_0$ is the horizontal line $s = 0$.
\end{enumerate}

The following result is an immediate consequence from \cref{thm:ROGSliceIsom2} by setting $\mathcal{F} = \varnothing$ and $\FF' = \FF_{\leq h}$.  

\begin{thm}[Slice Recovery Theorem]\label{thm:SliceRecovery1}
Let $X$ be a $G$-spectrum, and let $h \geq 1$.  On the $(V+t-s, s)$-graded page where $t-s \geq 0$, the map
\[\varphi_h: \SliceSS(X)  \longrightarrow \EF_{\leq h} \sm \SliceSS(X)\]
induces an isomorphism of spectral sequences on or above the line $\mathcal{L}_{h-1}^V$.  In the region where $t-s < 0$, $\varphi_h$ induces an isomorphism of spectral sequences on or above the horizontal line $s= \tau_V(\mathcal{F}_{\leq h})$.
\end{thm}

\begin{cor}\label{cor:sliceRecoveryExample1}
The map 
\[\varphi_1: \SliceSS(X) \longrightarrow \widetilde{E}G \wedge \SliceSS(X)\]
 induces an isomorphism of spectral sequences on or above the horizontal line $s = 0$ on all the $(V+t-s, s)$-graded pages.  
\end{cor}
\begin{proof}
This follows directly from the discussion after \cref{df:LineLVh-1} and \cref{thm:SliceRecovery1} by setting $h = 1$.
\end{proof}

\begin{construction}[Stratification tower]\rm \label{construction:Tower}
From \cref{df:orderFamilies}, it is clear that the families $\FF_{\leq h-1}$ and $\FF_{\leq h}$ differ only when there exists a subgroup $H \subset G$ of order $h$.  This observation, combined with \cref{thm:SliceRecovery1}, shows that the slice spectral sequence of $X$ can be divided into distinct regions, corresponding to different orders of subgroups of $G$.  As we transition from one region to the next, the behaviors of the slice differentials undergoes a change.  More precisely, let $1 = h_0 < h_1 < \cdots < h_k < |G|$ denote the different orders of subgroups of $G$.  By \cref{thm:SliceRecovery1}, the chain (\ref{eq:LocalizedSSchain}) gives rise to a tower of spectral sequences that stratifies the original slice spectral sequence.  

\begin{equation} \label{diagram:SliceSSTower}
\begin{tikzcd}
\SliceSS(X) \ar[rr, "\mathcal{L}_{0}^V"] \ar[rrd, swap, "\mathcal{L}_{h_1-1}^V"] \ar[rrddd, swap, bend right = 7, "\mathcal{L}_{h_{k-1}-1}^V"] \ar[rrdddd, bend right = 35, swap, "\mathcal{L}_{h_k-1}^V"]&& \widetilde{E} G \wedge \SliceSS(X) \ar[d]  \\
&& \EF_{\leq h_1} \wedge \SliceSS(X) \ar[d]  \\
&& \vdots \ar[d] \\ 
&& \EF_{\leq h_{k-1}} \wedge \SliceSS(X) \ar[d] \\
&& \EF_{\leq h_k} \wedge \SliceSS(X) 
\end{tikzcd}
\end{equation}
The distinct regions are separated by the lines $\mathcal{L}_{h_i-1}^V$ of slope $(h_i-1)$, $1 \leq i \leq k$. 
 As shown in \cref{thm:SliceRecovery1}, on the $(V+t-s, s)$-graded page where $t-s \geq 0$, all the slice differentials on or above the line $\mathcal{L}_{h_i-1}^V$ can be fully recovered from the localized slice spectral sequence $\EF_{\leq h_i} \wedge \SliceSS(X)$.  
\end{construction}

To gain deeper insights into the stratification tower (\ref{diagram:SliceSSTower}), we will establish the following result, which proves vanishing properties in the $RO(G)$-graded slice spectral sequence.  More precisely, we will show that classes are concentrated in certain conical regions.  Historically, these regions have only been explicitly identified for the integer-graded slice spectral sequence in \cite[Section~4]{HHR}.  

\begin{thm}[Positive Cone]\label{thm:SliceSSVanishingRegion}
Consider a connective $G$-spectrum $X$ and let $V \in RO(G)$.  Define $k(V)$ to be the smallest non-negative integer such that the representation ${\left(k(V)\cdot \rho_G + V\right)}$ becomes an actual $G$-representation (if $V$ is already an actual representation, then ${k(V) = 0}$).  On the $(V+t-s, s)$-graded $\mathcal{E}_2$-page of the slice spectral sequence of $X$, all the classes are concentrated within the region defined by the following inequalities:
\begin{align*}
s & \geq -|V|-k(V)\cdot|G|, \\
s & \leq (|G|-1)(t-s) - |V|+ |G|\cdot \max_{H \subseteq G} |V^H|. 
\end{align*}
\end{thm}

\begin{rmk}\rm 
\label{rmk:positiveCone}
The inequalities in \cref{thm:SliceSSVanishingRegion} define a conical region on each $RO(G)$-graded page.  This region is bounded by a horizontal line of filtration ${(-|V|-k(V)\cdot|G|)}$ and a line of slope ${(|G|-1)}$.  On the integer-graded page, where $V=0$ and ${k(V) = 0}$, this region is the cone bounded by the inequalities $s \geq 0$ and ${s \leq (|G|-1)(t-s)}$.  This cone is commonly referred to as the ``positive cone''.  Given this, we will use the term ``positive cone'' to describe the region defined on each $RO(G)$-graded page.  \cref{thm:SliceRecovery1} and \cref{construction:Tower} show that the tower (\ref{diagram:SliceSSTower}) further stratifies the positive cone into distinct regions separated by lines of slope $(|H|-1)$, where $H$ ranges over the subgroups of $G$.  
\end{rmk}

\begin{rmk}\rm
When $V = 0$, the stratification tower provides the cleanest stratification of the positive cone.  In this case, the positive cone is bounded by the inequalities $s \geq 0$ and $s \leq (|G|-1)(t-s)$, and the line $\mathcal{L}_{h_i-1}^0$ is the line of slope $(h_i-1)$ passing through the origin for all $0 \leq i \leq k$ (see remark after \cref{df:LineLVh-1}).  These lines all intersect at the origin and completely stratify the positive cone (see \cref{fig:PicStratificationTowerV=0}).
\begin{figure}
\begin{center}
\makebox[\textwidth]{\hspace{-0.5in}\includegraphics[trim={0cm 0cm 0cm 0cm}, clip, scale = 0.6]{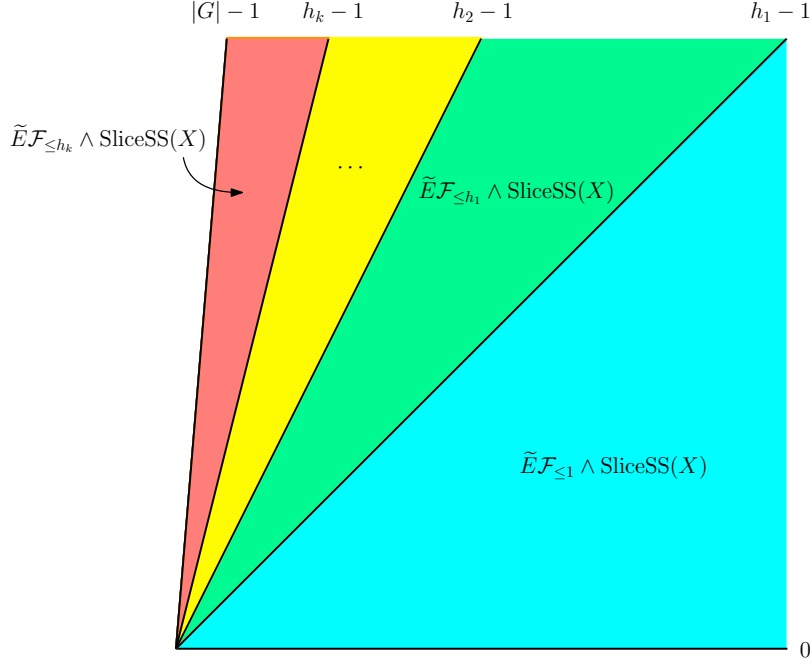}}
\caption{Stratification regions of the slice spectral sequence when $V = 0$.}
\hfill
\label{fig:PicStratificationTowerV=0}
\end{center}
\end{figure}
\end{rmk}

\begin{rmk}\rm
In general, for any $V \in RO(G)$, the positive cone is bounded by the inequalities 
\begin{align*}
s & \geq -|V|-k(V)\cdot|G| \\
s & \leq (|G|-1)(t-s) - |V|+ |G|\cdot \max_{H \subseteq G} |V^H|
\end{align*}
as proven in \cref{thm:SliceSSVanishingRegion}.  The localized slice spectral sequence $\EF_{\leq 1} \wedge \SliceSS(X)$ will recover everything in $\SliceSS(X)$ that is on or above the line $\mathcal{L}_0^V$, which is the horizontal line $s = 0$ (see \cref{fig:PicStratificationTower}).  It is worth noting that, unlike the $V = 0$ case, the lines $\mathcal{L}_{h_i-1}^V$ for $0 \leq i \leq k-1$ won't necessarily all intersect at the same point.

Furthermore, it is interesting to observe that in this scenario, the localized slice spectral sequences do not completely capture all the information present in the original slice spectral sequence.  More specifically, they do not capture the information between the lines $s = -|V|-k(V)\cdot|G|$ and $s = 0$ (the gray region in \cref{fig:PicStratificationTower}).  Understanding the behavior of the slice spectral sequence in this region requires using the generalized Tate diagram of the equivariant slice filtration, a topic that will be explored in an upcoming paper with Yutao Liu and Guoqi Yan. 

\begin{figure}
\begin{center}
\makebox[\textwidth]{\includegraphics[trim={0cm 0cm 0cm 0cm}, clip, scale = 0.6]{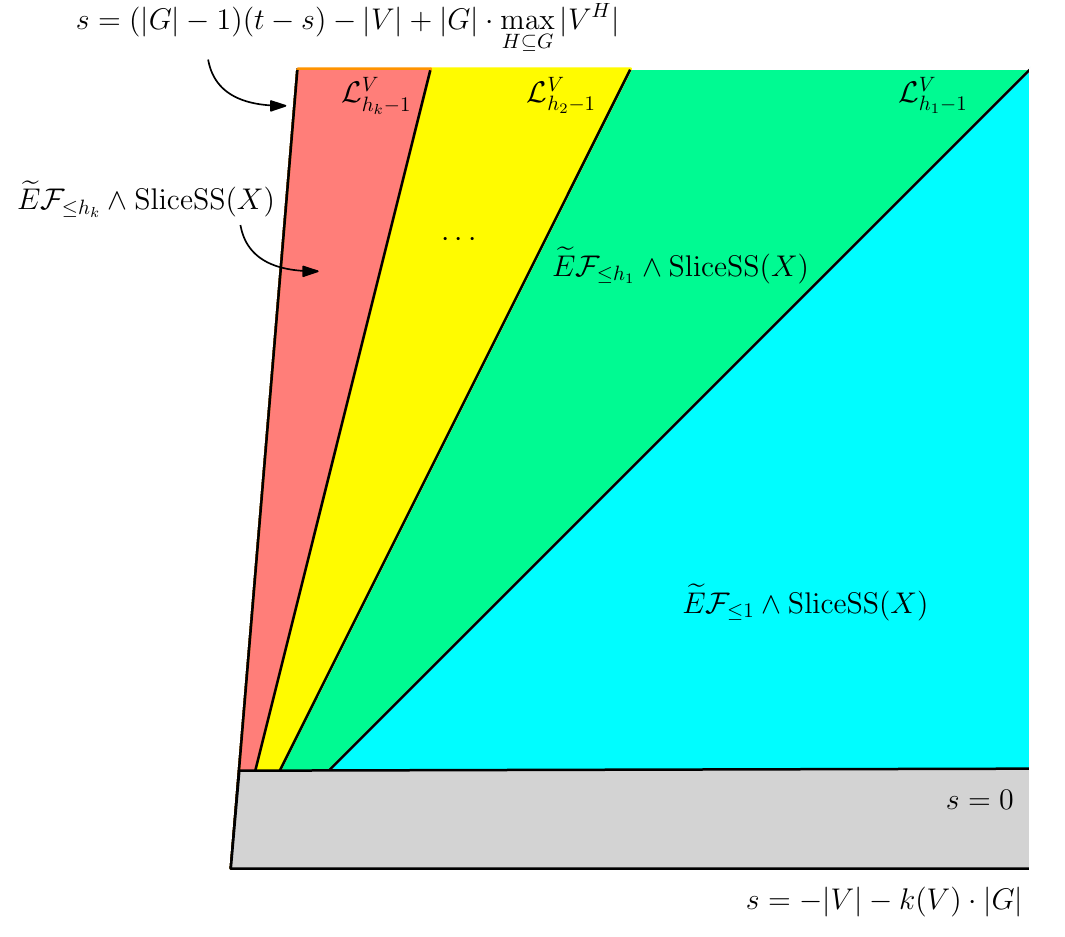}}
\caption{The stratification tower of the slice spectral sequence.}
\hfill
\label{fig:PicStratificationTower}
\end{center}
\end{figure}

\end{rmk}
\begin{lem}\label{lem:SliceCoconnectivity}
    Consider a nonnegative integer $n$ and a $G$-spectrum $Z$ which is slice $\leq n$. For every $G$-representation $W$, the $G$-spectrum $S^{-W} \wedge Z$ is slice $\leq n$, and $\Mpi_k(S^{-W} \wedge Z) = 0$ for $k>n$. 
\end{lem}
\begin{proof}
For any slice cell $\widehat{S}$ of dimension greater than $n$, the smash product $\widehat{S} \wedge S^{W}$ is slice $>n$ by \cite[Proposition 4.26]{HHR}.  Here, we have used the fact that $S^W$ is of slice $\geq 0$, which is equivalent to being connective.  Therefore, we have the equality 
\begin{align*}
\left[\widehat{S}, S^{-W} \wedge Z\right]^G \cong
\left[\widehat{S} \wedge S^W, Z\right]^G = 0.
\end{align*}
It follows from \cite[Lemma~4.14]{HHR} that the $G$-spectrum 
$S^{-W} \wedge Z$ is slice $\leq n$.

By \cite[Propostion~4.40]{HHR}, the $G$-cell $G/H_+ \wedge S^k$ is slice $\geq k$ whenever $k \geq 0$. Thus, $\pi_k^H(S^{-W} \wedge Z) = [G/H_+ \wedge S^k, S^{-W} \wedge Z]$ vanishes for every subgroup $H\subset G$. 
\end{proof}

\begin{proof}[Proof of \cref{thm:SliceSSVanishingRegion}]
The classes on the $\mathcal{E}_2$-page that are in position $(V+t-s, s)$ are elements in the Mackey functor homotopy group
\[\Mpi_{V+t-s} P^{|V|+t}_{|V|+t} X = \Mpi_{t-s} \left(S^{-V} \wedge P^{|V|+t}_{|V|+t} X\right).\]
To prove our claim, it suffices to show that this Mackey functor vanishes whenever the inequalities 
\[s < -|V|-k(V)\cdot|G|\] 
or
\[s > (|G|-1)(t-s) - |V|+ |G|\cdot \max_{H \subseteq G} |V^H|\]
hold.  

To prove the first inequality, we will analyze the slice coconnectivity of the $G$-spectrum $S^{-V} \wedge P^{|V|+t}_{|V|+t} X$. By the definition of $k(V)$, the formula $V' = k(V) \cdot \rho_G + V$.   defines an actual $G$-representation, which satisfies $S^{-V} \simeq S^{-V'} \wedge S^{k(V) \cdot \rho_G}$.  We have the equivalence
\begin{align*}
S^{-V} \wedge P^{|V|+t}_{|V|+t} X &\simeq S^{-V'} \wedge S^{k(V) \cdot \rho_G} \wedge P^{|V|+t}_{|V|+t} X \\
{} &\simeq  S^{-V'} \wedge P^{|V|+t+k(V) \cdot |G|}_{|V|+t+k(V) \cdot |G|} (S^{k(V) \cdot \rho_G} \wedge X).
\end{align*}
Here, the second equivalence follows from the fact that smashing by regular representation spheres induces an equivalence of slices (see \cite[Corollary~4.25]{HHR}).

% For any slice cell $\widehat{S}$ of dimension greater than $(|V| + t + k(V) \cdot |G|)$, the smash product $\widehat{S} \wedge S^{V'}$ is of slice $>(|V| + t + k(V) \cdot |G|)$ by \cite[Proposition 4.26]{HHR}.  Here, we have used the fact that $S^{V'}$ is of slice $\geq 0$, which is equivalent to being connective.  Therefore, we have the equality 
% \begin{align*}
% &\quad \, \left[\widehat{S}, S^{-V'} \wedge P^{|V|+t+k(V) \cdot |G|}_{|V|+t+k(V) \cdot |G|} (S^{k(V) \cdot \rho_G} \wedge X)\right]^G \\
% &=\left[\widehat{S} \wedge S^{V'}, P^{|V|+t+k(V) \cdot |G|}_{|V|+t+k(V) \cdot |G|} (S^{k(V) \cdot \rho_G} \wedge X)\right]^G \\
% &= 0.
% \end{align*}
% It follows from \cite[Lemma~4.14]{HHR} that the $G$-spectrum 
% \[S^{-V'} \wedge P^{|V|+t+k(V) \cdot |G|}_{|V|+t+k(V) \cdot |G|} (S^{k(V) \cdot \rho_G} \wedge X)\] 
% is of slice $\leq (|V| + t + k(V) \cdot |G|)$.

Since $X$ is connective, $|V|+t$ is necessarily non-negative in our analysis (or else $P^{|V|+t}_{|V|+t} X$ will be contractible and its homotopy groups will automatically vanish).  Since $k(V)$ is also non-negative by definition, ${|V| + t + k(V) \cdot |G| \geq 0}$. % By \cite[Propostion~4.40]{HHR}, the $G$-cell $G/H_+ \wedge S^{t-s}$ is of slice $\geq t-s$ whenever $t-s \geq 0$.  Therefore, our analysis shows that the homotopy group
Thus, \cref{lem:SliceCoconnectivity} implies that 
\[\Mpi_{t-s} \left(S^{-V} \wedge P^{|V|+t}_{|V|+t} X\right) = \Mpi_{t-s} \left(S^{-V'} \wedge P^{|V|+t+k(V) \cdot |G|}_{|V|+t+k(V) \cdot |G|} (S^{k(V) \cdot \rho_G} \wedge X)\right)\]
vanishes whenever the inequality 
\[t-s > |V| + t + k(V) \cdot |G|\]
or equivalently,
\[s < -|V| - k(V) \cdot |G|\]
holds.  This proves the first inequality.  

For the second inequality, we will analyze the connectivity of $S^{-V} \wedge P^{|V|+t}_{|V|+t} X$.  For any subgroup $H \subset G$, we have the equivalence 
\begin{align*}
\Phi^H\left(S^{-V} \wedge P^{|V|+t}_{|V|+t} X\right) &\simeq \Phi^H(S^{-V}) \wedge \Phi^H\left(P^{|V|+t}_{|V|+t} X\right) \\ 
&\simeq S^{-|V^H|} \wedge \Phi^H\left(P^{|V|+t}_{|V|+t} X\right). 
\end{align*}
Since $X$ is connected, we may assume again that $|V|+t \geq 0$.  By \cite[Theorem~2.5]{HillYarnall}, the $H$-geometric fixed points above is $\left(-|V^H| + \left\lceil \frac{|V|+t}{|H|}\right\rceil\right)$-connective.  Therefore, in order to find a lower bound for the connectivity of $S^{-V} \wedge P^{|V|+t}_{|V|+t} X$, it suffices to find a lower bound for 
\[\min_{H \subseteq G} \left(-|V^H| + \left\lceil \frac{|V|+t}{|H|}\right\rceil\right).\]
We have the following inequalities:
\begin{align*}
\min_{H \subseteq G} \left(-|V^H| + \left\lceil \frac{|V|+t}{|H|}\right\rceil\right) & \geq \min_{H \subseteq G} \left(-|V^H|\right) + \min_{H \subseteq G} \left\lceil \frac{|V|+t}{|H|}\right\rceil \\ 
&\geq -\max_{H \subseteq G} |V^H| + \left\lceil \frac{|V|+t}{|G|}\right\rceil  \quad \left(|V|+t \geq 0\right)\\ 
&\geq -\max_{H \subseteq G} |V^H| + \frac{|V|+t}{|G|}.
\end{align*}
Therefore $\left(S^{-V} \wedge P^{|V|+t}_{|V|+t} X\right)$ is at least $\left(-\max_{H \subseteq G} |V^H| + \frac{|V|+t}{|G|}\right)$-connective, and $\Mpi_{t-s} \left(S^{-V} \wedge P^{|V|+t}_{|V|+t} X\right)$ will vanish when the inequality 
\[t-s <-\max_{H \subseteq G} |V^H| + \frac{|V|+t}{|G|} \]
or equivalently, 
\[s> (|G|-1)(t-s) - |V|+ |G|\cdot \max_{H \subseteq G} |V^H|\]
holds.  This proves the second inequality. 
\end{proof}

%%%%%%%

\section{Specialization to cyclic \texorpdfstring{$p$}{text}-groups}
\label{sec:SpecializationCyclic}

In this section, we will specialize our results from \cref{sec:LocSliceSS} and \cref{sec:Stratification} to the case when $G$ is a cyclic $p$-group.  

\begin{df}\rm \label{df:FFH}
For $G$ a finite group and $H \subset G$ a subgroup, let $\FF[H]$ denote the family consisting of all subgroups of $G$ that do not contain $H$. 
\end{df}

When $G = C_{p^n}$, the subgroups of $G$ are linearly ordered by inclusion.  More precisely, they are of the form $C_{p^k}$ for $0 \leq k \leq n$.  For all $k$ and $\ell$ such that $1 \leq k \leq n$ and $p^{k-1}\leq \ell \leq p^k-1$, we have the equality $\mathcal{F}_{\leq \ell} = \mathcal{F}[C_{p^k}]$.  The chain of inclusions 
\[\FF_{\leq 0} \subset \FF_{\leq 1} \subset \cdots \FF_{\leq |G|-1}\]
is essentially 
\[\varnothing \subset \mathcal{F}[C_p] \subset \mathcal{F}[C_{p^2}] \subset \cdots \subset \mathcal{F}[C_{p^n}].\]

For $0 \leq i \leq n-1$, let $\lambda_{{i}}$ denote the 2-dimensional real $C_{p^n}$-representation corresponding to rotation by $\left(\frac{2\pi}{p^{i+1}}\right)$.  In particular, when $p = 2$ and $i = 0$, the representation $\lambda_0$ corresponds to rotation by $\pi$ and is equal to $2\sigma$, where $\sigma$ is the real sign representation of $C_{2^n}$.  If $X$ is a $C_{p^n}$-spectrum, then for all $1 \leq k \leq n$,  $\EF[C_{p^k}] \simeq S^{\infty \lambda_{n-k}}$, and there is an equivalence 
\[\EF[C_{p^k}] \wedge X \simeq S^{\infty \lambda_{{n-k}}} \wedge X \simeq a_{\lambda_{{n-k}}}^{-1}X.\]
Here, $a_{\lambda_{{n-k}}}: S^0 \to S^{\lambda_{{n-k}}}$ is the Euler class, corresponding to the inclusion of $S^0=\{0, \infty\}$ into $S^{\lambda_{n-k}}$.  Given this equivalence, we will sometimes denote the localized slice spectral sequence $\EF[C_{p^k}] \wedge \SliceSS(X)$ as $a_{\lambda_{n-k}}^{-1}\SliceSS(X)$.

For any $G$-spectrum $X$, we have the equivalence
\[(\EF[H] \wedge X)^H \simeq \Phi^H(X).\]
In the special case when $G = C_{p^n}$ and $H= C_{p^k}$, there is a residual $(C_{p^n}/C_{p^k})$-action on the $C_{p^k}$-geometric fixed points $\Phi^{C_{p^k}}(X)$, and we have the equivalence 
\[(\EF[C_{p^k}] \wedge X)^{C_{p^n}} \simeq  \left(\Phi^{C_{p^k}}(X)\right)^{C_{p^n}/C_{p^k}}.\]

Let $k$ and $\ell$ be integers such that $0 \leq k < \ell \leq n$.  If we set $\mathcal{F} = \mathcal{F}[C_{p^k}]$ and $\mathcal{F}'=\mathcal{F}[C_{p^\ell}]$ in the Comparison Theorem (\cref{thm:ROGSliceIsom2}), then for a connective $G$-spectrum $X$, the map of spectral sequences (left vertical map)
\[\begin{tikzcd}
\EF[C_{p^k}] \wedge \SliceSS(X) \ar[r, Rightarrow] \ar[d] & \pi_*^{C_{p^n}/C_{p^k}} \Phi^{C_{p^k}}(X) \ar[d] \\
\EF[C_{p^\ell}] \wedge \SliceSS(X) \ar[r, Rightarrow] &  \pi_*^{C_{p^n}/C_{p^\ell}} \Phi^{C_{p^\ell}}(X)
\end{tikzcd}\]
induces an isomorphism of spectral sequences on or above the line $\mathcal{L}_{p^{\ell-1}-1}$ in the integer-graded page.  Furthermore, if we set $k = 0$, the Slice Recovery Theorem (\cref{thm:SliceRecovery1}) shows that for the map of spectral sequences 
\[\begin{tikzcd}
\SliceSS(X) \ar[r, Rightarrow] \ar[d, "\varphi_{p^{\ell-1}}"] & \pi_*^{C_{p^n}} X \ar[d] \\
\EF[C_{p^\ell}] \wedge \SliceSS(X) \ar[r, Rightarrow] &  \pi_*^{C_{p^n}/C_{p^\ell}} \Phi^{C_{p^\ell}}(X),
\end{tikzcd}\]
the localized slice spectral sequence $\EF[C_{p^\ell}] \wedge \SliceSS(X)$ recovers all the differentials in $\SliceSS(X)$ that originates on or above the line $\mathcal{L}_{p^{\ell-1}-1}$ in the integer-graded page.  When we consider these spectral sequences as $RO(G)$-graded spectral sequences, \cref{thm:SliceRecovery1} ensures that analogous statements hold for all $(V+t-s, s)$-graded pages as well by replacing the lines $\mathcal{L}_{p^{\ell-1}-1}$ with $\mathcal{L}_{p^{\ell-1}-1}^V$.  

In summary, when $G= C_{p^n}$ and $X$ is connective, the stratification tower (\ref{diagram:SliceSSTower}) is the following tower: 

\begin{equation} \label{diagram:SliceSSTower2}
\begin{tikzcd}
\SliceSS(X) \ar[rr, "\mathcal{L}^V_0"] \ar[rrd, "\mathcal{L}^V_{p-1}"] \ar[rrddd, bend right = 10, "\mathcal{L}^V_{p^{n-2}-1}"] \ar[rrdddd, bend right = 30, "\mathcal{L}^V_{p^{n-1}-1}",swap]&& a_{\lambda_{n-1}}^{-1} \SliceSS(X) \ar[d, "\mathcal{L}^V_{p-1}"] \ar[r, Rightarrow] & \Mpi_\star \Phi^{C_p}(X) \ar[d] \\
&& a_{\lambda_{n-2}}^{-1} \SliceSS(X) \ar[d, "\mathcal{L}^V_{p^2-1}"] \ar[r,Rightarrow] & \Mpi_\star \Phi^{C_{p^2}}(X) \ar[d]\\
&& \vdots \ar[d, "\mathcal{L}^V_{p^{n-2}-1}"] & \vdots \ar[d] \\ 
&& a_{\lambda_1}^{-1} \SliceSS(X) \ar[d, "\mathcal{L}^V_{p^{n-1}-1}"] \ar[r, Rightarrow] & \Mpi_\star \Phi^{C_{p^{n-1}}}(X) \ar[d] \\
&& a_\sigma^{-1} \SliceSS(X) \ar[r,Rightarrow] & \pi_*\, \Phi^{C_{p^n}}(X).
\end{tikzcd}
\end{equation}
In the tower above, the maps are labeled by the lines $\mathcal{L}^V_{h-1}$, indicating the isomorphism regions.  The slice spectral sequence of $X$ is then stratified into different regions separated by the lines $\mathcal{L}^V_{p-1}$, $\mathcal{L}^V_{p^2-1}$, $\ldots$, and $\mathcal{L}^V_{p^{n-1}-1}$.  The differentials within these regions can be recovered from the localized slice spectral sequences, which compute the geometric fixed points equipped with the residue $(C_{p^n}/C_{p^k})$-action.  As we move up the tower, \cref{thm:SliceRecovery1} shows that each localized slice spectral sequence contains increasingly more information about the original slice spectral sequence. 

This stratification provides an inductive approach for computing the slice spectral sequence of $X$.  We start by computing the bottom spectral sequence of the tower, {$\EF[C_{p^n}] \wedge \SliceSS(X) = a_{\lambda_0}^{-1}\SliceSS(X)$}, which recovers all the differentials in $\SliceSS(X)$ originating on or above the line $\mathcal{L}^V_{p^{n-1}-1}$.  Then, we proceed to compute the next spectral sequence, ${\EF[C_{p^{n-1}}] \wedge \SliceSS(X) = a_{\lambda_1}^{-1}\SliceSS(X)}$, which recovers the differentials in $\SliceSS(X)$ originating on or above the line $\mathcal{L}^V_{{p^{n-2}}-1}$.  This process continues until we have computed the topmost spectral sequence, ${\EF[C_p] \wedge \SliceSS(X) = a_{\lambda_{{n-1}}}^{-1}\SliceSS(X)}$.  The differentials in this localized slice spectral sequence recover all of the original slice differentials, which are concentrated in the first quadrant.

\begin{exam}\rm
Consider the $C_2$-spectrum $\BPR \langle 1 \rangle$.  In order to compute its slice differentials, we first compute the differentials in the localized slice spectral sequence 
 $\widetilde{E}C_2 \wedge \SliceSS(\BPR \langle 1 \rangle) = a_\sigma^{-1}\SliceSS(\BPR \langle 1 \rangle)$ (see \cref{fig:BPRoneSigma}). 
 ~\cref{thm:SliceRecovery1} then implies that all the differentials in $\SliceSS(\BPR \langle 1 \rangle)$ (see \cref{fig:BPRone}) can be recovered from the differentials in $a_\sigma^{-1}\SliceSS(\BPR \langle 1 \rangle)$ by truncating at the line $\mathcal{L}_0$.
\end{exam}

\begin{figure}
\begin{center}
\makebox[\textwidth]{\includegraphics[trim={2.5cm 15.7cm 4cm 4cm}, clip, page = 1, scale = 0.7]{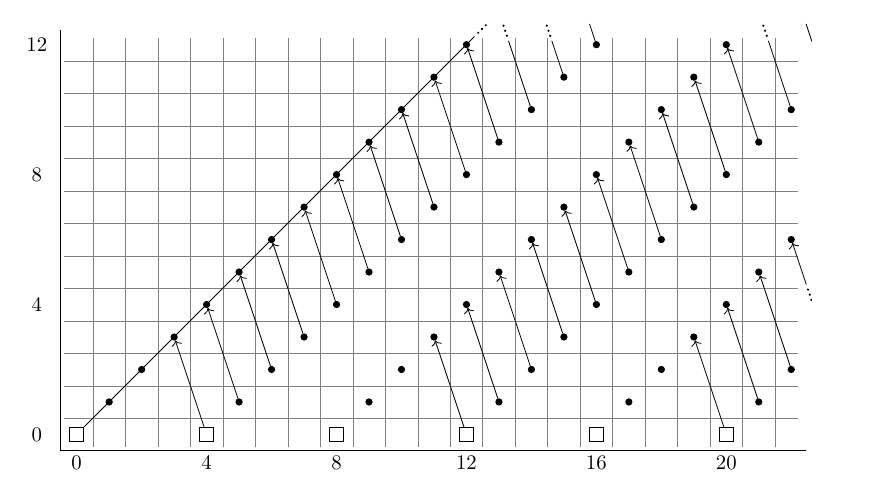}}
\caption{The slice spectral sequence of $\BPR\langle 1 \rangle$.} 
\hfill
\label{fig:BPRone}
\end{center}
\end{figure}

\begin{figure}
\begin{center}
\makebox[\textwidth]{\includegraphics[trim={2.5cm 9.5cm 4cm 4cm}, clip, page = 2, scale = 0.7]{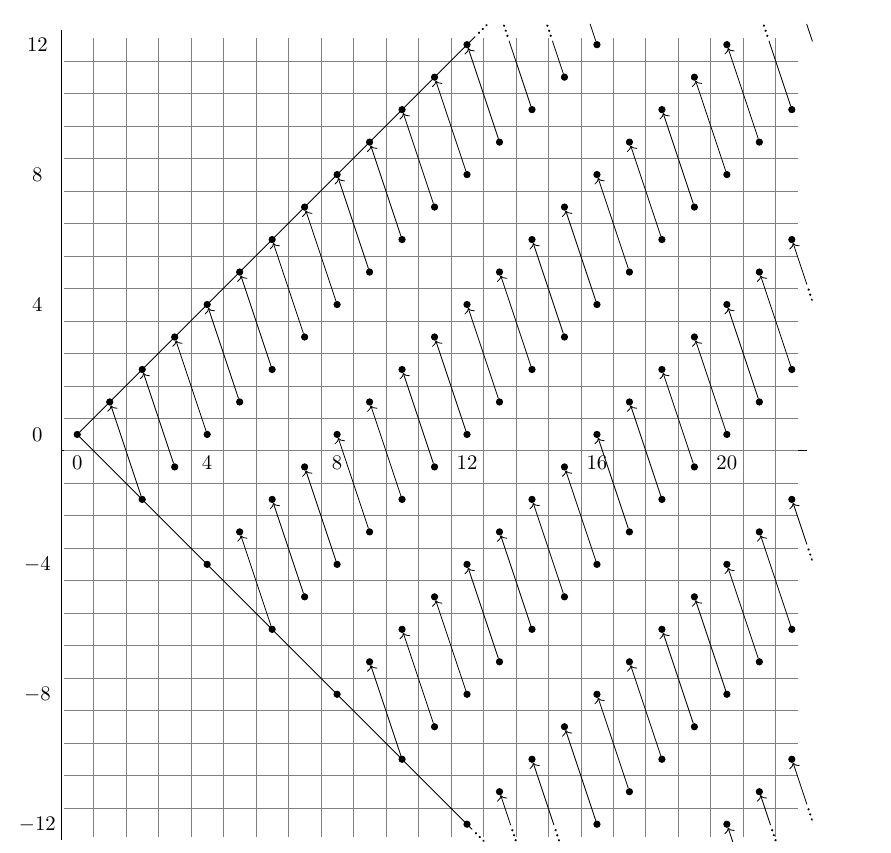}}
\caption{The localized slice spectral sequence of $a_\sigma^{-1}\BPR\langle 1 \rangle$.} 
\hfill
\label{fig:BPRoneSigma}
\end{center}
\end{figure}

\begin{exam} \rm
Consider the $C_4$-spectrum $\BPCfourone$.  There are two localized slice spectral sequences:
\[\EF[C_4]\wedge \SliceSS(\BPCfourone) = a_\sigma^{-1}\SliceSS(\BPCfourone)\]
and
\[\EF[C_2] \wedge \SliceSS(\BPCfourone) = a_{\lambda_1}^{-1}\SliceSS(\BPCfourone).\] 
We first localize with respect to $a_\sigma$ and compute $a_\sigma^{-1}\SliceSS(\BPCfourone)$ (see \cref{fig:BPCfouroneSigma}).  This recovers all the differentials on or above the line $\mathcal{L}_1$ in the original slice spectral sequence $\SliceSS(\BPCfourone)$ (see ~\cref{fig:BPCfourone}).  Then, we localize with respect to $a_{\lambda_1}$ and compute $a_{\lambda_1}^{-1}\SliceSS(\BPCfourone)$ (see ~\cref{fig:BPCfouroneLambda}).  Truncating this spectral sequence at the line $\mathcal{L}_0$ will then recover all the differentials in $\SliceSS(\BPCfourone)$.  
\end{exam}

\begin{figure}
\begin{center}
\makebox[\textwidth]{\includegraphics[trim={0cm 3.6cm 4cm 4cm}, clip, page = 3, scale = 0.7]{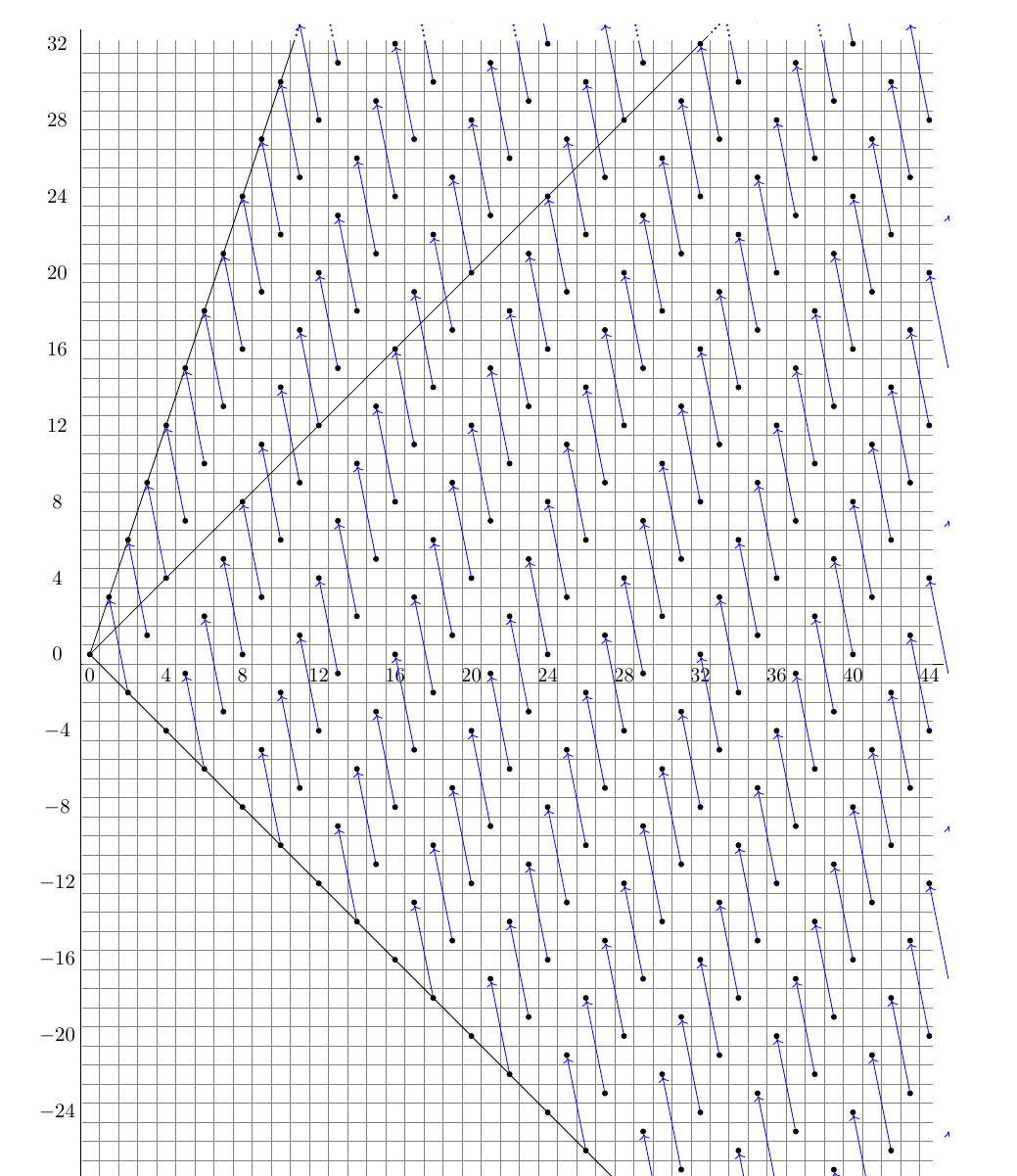}}
\caption{The localized slice spectral sequence of $a_\sigma^{-1}\BPCfourone$.} 
\hfill
\label{fig:BPCfouroneSigma}
\end{center}
\end{figure}

\begin{figure}
\begin{center}
\makebox[\textwidth]{\includegraphics[trim={0cm 12cm 4cm 4cm}, clip, page = 1, scale = 0.7]{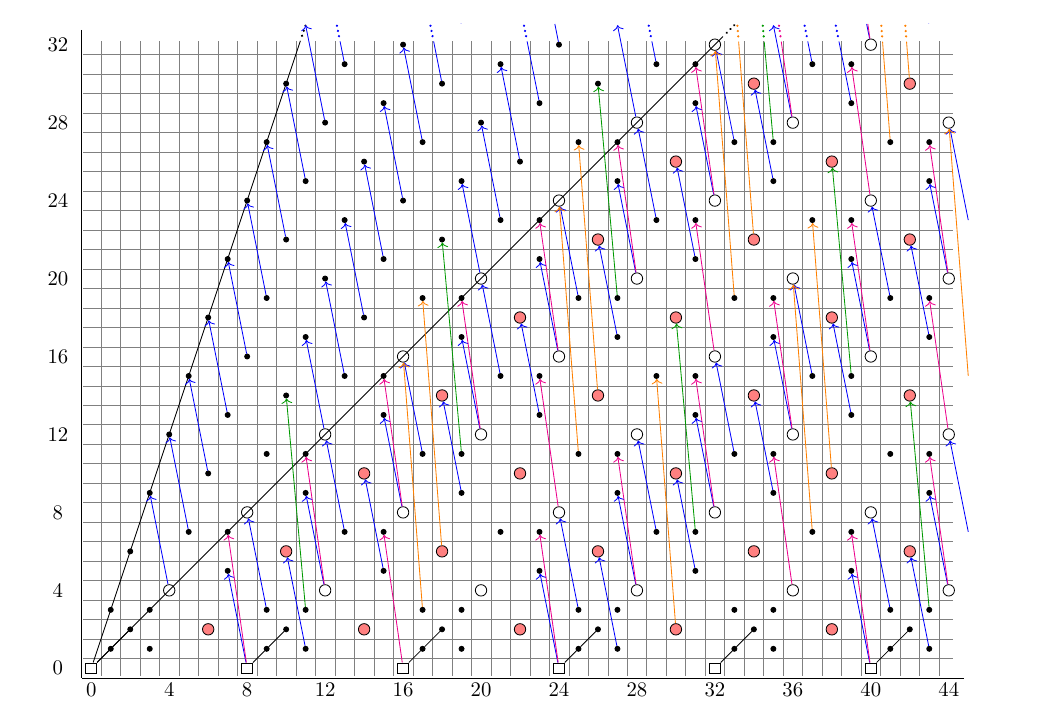}}
\caption{The slice spectral sequence of $\BPCfourone$.} 
\hfill
\label{fig:BPCfourone}
\end{center}
\end{figure}

\begin{figure}
\begin{center}
\makebox[\textwidth]{\includegraphics[trim={0cm 3.6cm 4cm 4cm}, clip, page = 2, scale = 0.7]{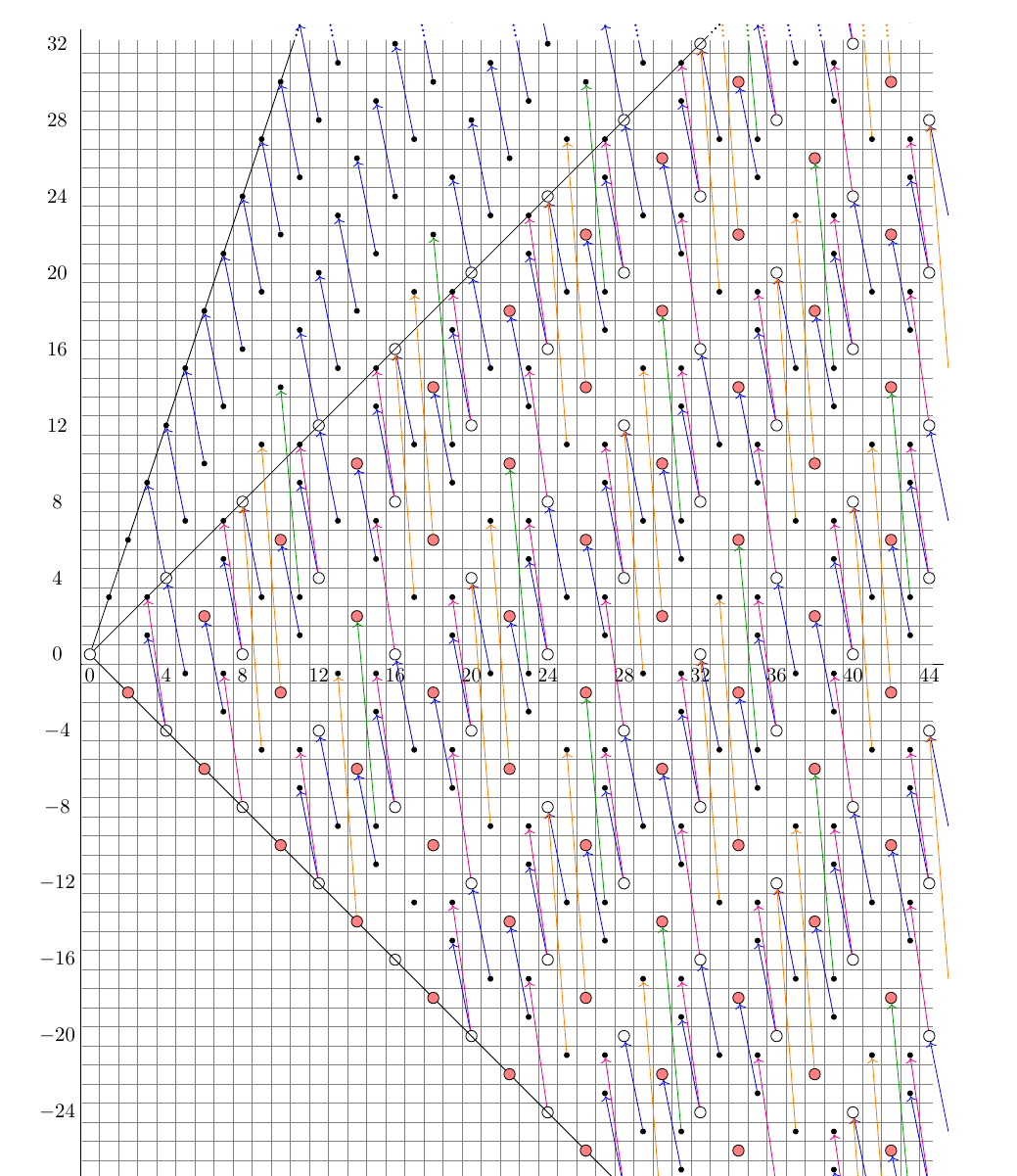}}
\caption{The localized slice spectral sequence of $a_{\lambda_1}^{-1}\BPCfourone$.} 
\hfill
\label{fig:BPCfouroneLambda}
\end{center}
\end{figure}

%%%%%%%%%%%%%%%%%%%%%%%%%%%%%%%%%%%%
%%%%%%%%%%%%%%%%%%%%%%%%%%%%%%%%%%%%

%%%%%%%%%%%%%%%%%%%%
\bibliographystyle{alpha}
\bibliography{Bibliography}

\end{document}